\documentclass{article}
  
\pdfoutput=1
\usepackage{graphicx}
\usepackage[noblocks]{authblk}
\usepackage{mathtools,amsmath,amssymb,amsfonts,amsthm}
\usepackage{mathalfa}

\usepackage{tikz}
\usetikzlibrary{calc}
\usepackage{color}
\usepackage{url}
\usepackage{fullpage}

\usepackage{enumitem}

\usepackage[english]{babel}
\usepackage{sublabel}
\usepackage{tabularx}

\usepackage[boxruled]{algorithm2e}

\usepackage{xcolor}
 \usepackage{caption}
 \usepackage{subcaption}
\usepackage{verbatim}
\usepackage{float}

\usepackage{setspace}

\usepackage{ifthen}

\newtheoremstyle{mythmstyle}{}{}{}{}{\bf}{.}{ }{}

\theoremstyle{plain}
\newtheorem{theorem}{Theorem}
\newtheorem{lemma}{Lemma}
\newtheorem{proposition}{Proposition}
\newtheorem{corollary}{Corollary}

\theoremstyle{definition}

\newtheorem{definition}{Definition}

\newtheorem{step}{Step}

\theoremstyle{mythmstyle}
\newtheorem{example}{Example}
\newtheorem{remark}{Remark}

\definecolor{darkgreen}{rgb}{0.0, 0.3, 0.13}
\newcommand{\virgolette}[1]{``#1''}
\newcommand{\R}{\mathbb{R}}
\newcommand{\N}{\mathbb{N}}

\newcommand{\cC}{\mathcal{C}}
\newcommand{\cF}{\mathcal{F}}
\newcommand{\cD}{\mathcal{D}}
\newcommand{\cK}{\mathcal{K}}
\newcommand{\cY}{\mathcal{Y}}
\newcommand{\cI}{\mathcal{I}}

\newcommand{\cN}{\mathcal{N}}
\newcommand{\cX}{\mathcal{X}}
\newcommand{\cR}{\mathcal{R}}
\newcommand{\cS}{\mathcal{S}}

\newcommand{\cU}{\mathcal{U}}
\newcommand{\cV}{\mathcal{V}}

\newcommand{\cO}{\mathcal{O}}

\newcommand{\cPd}{\mathcal{PD}}
\newcommand{\cPQ}{\mathcal{PQ}}

\newcommand{\xqed[1]}{%
  \leavevmode\unskip\penalty9999 \hbox{}\nobreak\hfill
  \quad\hbox{#1}}

\newcommand{\cLL}{\mathcal{L}_F}
\newcommand{\cL}{\mathcal{L}}
\newcommand{\cA}{\mathcal{A}}
\newcommand{\ball}{{I\kern -.35em B}}
\newcommand{\dom}{\mathop{\rm dom}\nolimits}

\newcommand{\meas}{\textnormal{meas}}

\DeclareMathOperator{\Mm}{\mathbf{Mm}}
\DeclareMathOperator{\co}{co}
\DeclareMathOperator{\Sym}{Sym}
\DeclareMathOperator{\bd}{bd}
\DeclareMathOperator{\inn}{int}
\DeclarePairedDelimiterX{\inp}[2]{\langle}{\rangle}{#1, #2}
  
\begin{document}

\title{Piecewise structure of Lyapunov functions and densely checked decrease conditions for hybrid systems}
\author[1]{Matteo Della Rossa\thanks{{\tt mdellaro@laas.fr}}}
\author[2]{Rafal Goebel}
\author[1]{Aneel Tanwani}
\author[1,3]{Luca Zaccarian}
\affil[1]{LAAS-CNRS, University of Toulouse, CNRS, Toulouse, France}
\affil[2]{Department of Statistics and Mathematics, Loyola University, Chicago IL, USA.}
\affil[3]{L.Zaccarian is also with Department of Industrial Engineering, University of Trento, Italy.}

\renewcommand\Authands{ and }
\maketitle

\begin{abstract}
We propose a class of locally Lipschitz functions with piecewise structure for use as Lyapunov functions for hybrid dynamical systems. 
Subject to some regularity of the dynamics, we show that Lyapunov inequalities can be checked only on a dense set and thus we avoid checking
them at points of nondifferentiability of the Lyapunov function. Connections to other classes of locally Lipschitz or piecewise regular 
functions are also discussed and applications to hybrid dynamical systems are included. 
\end{abstract}

\section{Introduction}
{\let\thefootnote\relax\footnotetext{Work supported in part by ANR via grant {\sc HANDY} (number ANR-18-CE40-0010).}}
We study dynamical systems described by a differential inclusion and constrained to a set $\cC\subset\R^n$, that is,
 \begin{equation}\label{eq:DiffIntroduction}
\dot x\in F(x),\;\;\;\;x\in \cC,
\end{equation}
where $F:\R^n\rightrightarrows \R^n$ is a multivalued map. Differential inclusions generalize ordinary differential equations and model a variety of physical phenomena; see~\cite{AubCel84},~\cite{Dei92} and \cite{Smi02} for a formal introduction.
Remarkably, this paradigm is useful in the analysis of two important objects in control theory:
\begin{itemize}[leftmargin=*]
\item A differential equation with inputs $\dot x=f(x,u)$, $u\in U\subset \R^m$, can be modeled as a differential inclusion
\begin{equation}\label{eq:COntrolDiffIncl}
\dot x \in F^U(x):=\{f(x,u)\;\vert\;u\in U\},
\end{equation}
where the parameter $u\in U$ is the control input or an external disturbance.
\item  A differential equation $\dot x=f(x)$ with a discontinuous $f:\R^n \to \R^n$ can be, in a sense, regularized by considering a differential inclusion
\begin{equation}\label{eq:FilippovIntroduction}
\dot x\in F^{\text{Fil}}_f(x),
\end{equation}
where $F^{\text{Fil}}_f:\R^n\rightrightarrows \R^n$ is the so-called Filippov regularization, representing an enlargement of $f$. See~\cite{filippov1988differential} and the survey~\cite{cortes}.
\end{itemize} 
We are particularly interested in differential inclusions constrained to a pre-specified set $\cC\subset \R^n$. For example, in the framework of \emph{hybrid dynamical systems}~\cite{GoebSanf12}, the set $\cC$ represents the so called \emph{flow set}, that is the set where the solutions are allowed to ``flow'' (instead of jumping), as we discuss in what follows.

Lyapunov tools for stability analysis of equilibria, or compact, or just closed sets for differential inclusions have been the subject of extensive research. In particular, in the case $\cC=\R^n$, under some conditions on the map $F$, the existence of \emph{smooth}  Lyapunov functions is necessary and sufficient for asymptotic stability of a compact set, see \cite{CLARKE199869}, \cite{TeelPral00b}, \cite{Lin1996}. For the case where $\cC$ is a closed subset of $\R^n$, a converse  Lyapunov theorem (in the more general hybrid framework) could be found in \cite[Theorem 7.31]{GoebSanf12}. Nevertheless, it often happens that a \emph{nonsmooth} function $V$ may be easier to describe and construct.

As a first and remarkable example, one can consider a \emph{linear} differential inclusion (LDI), that is,
\[
\dot x\in \co\{A_ix\;\vert\;i\in \{1,\dots M\}\},
\]
for some $A_i \in \R^{n\times n}$, $i=1,\dots,M$. In this case, global asymptotic stability of the origin is equivalent to the existence of a \emph{smooth} Lyapunov function that is convex and homogeneous of degree 2, as shown in \cite{dayawansa}, \cite{molcha}. 
Despite this powerful theoretical result, such \emph{smooth} convex functions can be difficult to construct. Various classes of \emph{non-smooth} functions have been proposed to approximate them, for example maxima of quadratic functions and their convex conjugates \cite{goebel2}, \cite{goebel}, or functions with convex polyhedral level sets \cite{molcha}. Such functions are convex and locally Lipschitz but not necessarily differentiable, therefore the right notion of ``\emph{derivative of the Lyapunov function along  solutions}''  is crucial, see the book~\cite{clarke3}, or~\cite{ShePad94} and~\cite{BacCer99} for possible extensions. 

More generally, stability results for differential inclusions relying on nonsmooth Lyapunov functions  appear in \cite{ceragioli}, \cite{TeelPral00a}, \cite{ClarkeNONSmo}. 
In most of these works, the authors use of Clarke generalized gradient to formulate Lyapunov conditions at points where the Lyapunov function $V$ is not differentiable. Among other examples, Clarke generalized gradient is used in \cite{BaiGru12} for piecewise affine Lyapunov functions for state-dependent switching systems and in \cite{LibNes14} for interconnected hybrid systems. This strategy, effective for classical continuous-time nonlinear systems, was suggested in \cite[Chapter 4]{ClarkeNONSmo} and then well summarized in \cite[p.99]{TeelPral00a}.
The main drawback of Clarke generalized gradient condition is that it considers \emph{all} the scalar products between generalized gradients and \emph{all} values of the differential inclusion. Less restrictive conditions for stability of differential inclusions are presented in \cite{BacCer99}, under some regularity assumptions on the (nonsmooth) Lyapunov function.  
A different approach is presented in \cite{GraMic08}, where the nonsmooth function is studied by exploiting regularization-via-convolution techniques. 
One possible way to overcome the limitations of generalized gradients is to propose stability conditions that only need to be checked away from the points where the candidate Lyapunov function is not differentiable.

With this motivation in mind, we introduce in this paper a class of locally Lipschitz functions for which
the Lyapunov inequalities need to be checked only on a dense subset of $\cC$, under further hypotheses on the map $F:\cC\rightrightarrows \R^n$, but without further assumptions on $\cC$. 
The class includes functions that can be built using the pointwise maximum and the pointwise  minimum of continuously differentiable functions (see \cite{dellarossa19} and references therein) but is more general.
When $\cC$ is a closed set (possibly with an empty interior) we show that our conditions are less restrictive than the Clarke  gradient-based conditions. 
We also relate the proposed class to \emph{piecewise continuously differentiable functions}, introduced in \cite{Cha90} and \cite{Sch12} for optimization purposes. 

As the main application, we use the proposed class of Lyapunov functions for sufficient conditions for asymptotic stability in hybrid dynamical systems 
\begin{equation}\label{eq:IntroHybridSystem}
\mathcal{H}:\;
\begin{cases}
\dot x\in F(x),\;\;&\;x\in \cC,\\
x^+\in G(x),\;\;&\;x\in \cD;
\end{cases}
\end{equation}
see the book \cite{GoebSanf12} for the formal definition. 
Collecting our relaxed Lyapunov conditions on the \emph{flow} set $\cC$ with proper Lyapunov-\emph{jump} conditions on the set $\cD$, we propose a stability result that generalizes the standard sufficient Lyapunov conditions for hybrid systems in \cite[Theorem 3.18]{GoebSanf12}. This result allows revisiting a classical example from the reset control literature: the Clegg integrator in feedback with an integrating plant. This example has been shown to overcome intrinsic limitations of linear feedback systems in \cite{BekHol01}. For this example, both \cite{ZacNes11} and \cite{NesTeel2011} provided numerical and analytic nonconvex (and nonsmooth) Lyapunov functions. We give here a pair of new (arguably simpler) nonconvex functions, and also a new convex Lyapunov function.

Preliminary results in the direction of this article have been given in~\cite{DelGoe19}. Here, as compared to~\cite{DelGoe19}, we go further in the analysis of relaxed Lyapunov sufficient conditions for constrained differential inclusions, providing a deeper comparison with the existing literature on locally Lipschitz Lyapunov functions, and exploring the limitations of this approach with respect to the topological properties of the flow set $\cC$.

The paper is structured as follows. In Section~\ref{sec:Prelimi}, we give the basic definitions and recall the main results on locally Lipschitz Lyapunov functions, providing some example to illustrate the limitations of this approach. In Section~\ref{sec:main}, we present our main stability statements, while in Section~\ref{sec:Comparison} we deeply investigate the relations between our results and the existing literature on locally Lipschitz Lyapunov functions. In Section~\ref{sec:discFurth}, we introduce the concept of \emph{global} piecewise functions which simplifies our analysis. Finally, in Section~\ref{sec:HybridSystems} we apply all our previous results in the context of hybrid dynamical systems, presenting, as an example, the Clegg integrator.

\textbf{Notation:} For each $x\in \R^n$, $|x|$ denotes the usual Euclidean norm of $x$, and, given a closed set $\cA\subset\R^n $, $|x|_\cA$ denotes the Euclidean distance to $\cA$, that is $|x|_\cA:=\min_{y\in \cA}|x-y|$. The set of non-negative real numbers is defined by $\R_{\geq 0}:=\{x\in\R\;\vert\;x\geq 0\}$. A function $\alpha:\R\to \R_{\geq 0}$ is \emph{positive definite} ($\alpha\in\cPd$) if it is continuous, $\alpha(0)=0$, and $\alpha(s)>0$ if $ s\neq 0$. A function $\alpha:\R_{\geq 0}\to \R_{\geq 0}$ is of \emph{class $\cK$} ($\alpha \in \cK$) if it is continuous, $\alpha(0)=0$, and strictly increasing. It is of \emph{class $\cK_\infty$} if, in addition, it is unbounded. Given a Lebesgue-measurable set $A\subset \R^n$, with $\meas(A)$ we denote its Lebesgue measure. Given $X\subset\R^n$, $\inn(X)$, $\overline{X}$ and $\bd(X)$ denote respectively the interior, the closure and the boundary of $X$, while $\co(X)$ denotes the convex hull of $X$. Given $x\in \R^n$ and $\delta>0$, $\ball(x,\delta):=\{z\in \R^n\;\vert\;\; |x-z|<\delta\}$ is the open ball of radius $\delta$ centered at $x$.

\section{Everywhere and almost everywhere conditions}\label{sec:Prelimi}

\subsection{Background}

Let $\cC\subset \R^n$ be a set and $F:\cC\rightrightarrows \R^n$ a set-valued map. 
For details on the continuity concepts for $F$, recalled below, see \cite[Chapter 5]{rockafellar}.

\begin{definition}\label{def:Continuity}
\begin{itemize}[leftmargin=*]
\item $F$ is \emph{inner semicontinuous relative to $\cC$} at $x\in\cC$  if 
\[
F(x)\subset\liminf_{y\to^\cC x}F(y):=\{f\,\vert\,\forall x_k\to^\cC x\ \exists f_k\in F(x_k)\, \text{such that}\, f_k\to f\}.
\] (Above, $x_k\to^\cC x$ stands for $x_k\in\cC$ and $x_k\to x$ as $k\to \infty$). The map $F$ is \emph{inner semicontinuous in $\cC$} if it is inner semicontinuous relative to $\cC$ at each $x\in \cC$.
\item  $F$ is \emph{outer semicontinuous relative to $\cC$} at $x\in\cC$ if
\[
F(x)\supset\limsup_{y\to^\cC x}F(y):=\{f\,\vert\,\exists\,x_k\to^\cC x,\;\exists f_k\in F(x_k)\, \text{such that}\, f_k\to f\}.
\]
The map $F$ is \emph{outer semicontinuous in $\cC$} if it is outer semicontinuous relative to $\cC$ at each $x\in \cC$.

\item $F$ is \emph{locally bounded relative to $\cC$} at $x\in \cC$ if for some neighborhood $\cU$ of $x$ the set $F(\cU\cap\cC)\subset \R^n$ is bounded. It is \emph{locally bounded in $\cC$} if this holds at every $x\in \cC$.
\end{itemize}
\end{definition}
As an example, the map $F^U:\R^n\rightrightarrows \R^n$ in~\eqref{eq:COntrolDiffIncl} is outer semicontinuous, inner semicontinuous and locally bounded in $\R^n$ when $f:\R^n\times\R^m\to \R^n$ is continuous and $U\subset \R^m$ is compact. 
The map  $F^{\text{Fil}}_f:\R^n\rightrightarrows \R^n$ is outer semicontinuous and locally bounded in $\R^n$ when 
$f:\R^n\to\R^n$ is merely locally bounded, but $F^{\text{Fil}}_f$ usually fails to be inner semicontinuous at points where $f$ is discontinuous, as we underline in the subsequent Example~\ref{example:flower}.\\


\begin{definition}\label{def:SolDiffInc}
A \emph{solution} to \eqref{eq:DiffIntroduction} is an absolutely continuous function $\phi:\dom \phi:=[0,T_\phi)\to \R^n$, with $T_\phi>0$ (and possibly $T_\phi=+\infty$), such that
\[
\begin{aligned}
&\phi(t)\in \cC, \;\;\;\;\text{for all}\;t\in \inn(\dom \phi), \\
&\dot \phi(t)\in F(\phi(t)),\;\; \text{almost everywhere in}\; \dom \phi.
\end{aligned}
\]
We denote the set of all solutions of \eqref{eq:DiffIntroduction} by $\cS_{F,\cC}$. \hfill$\triangle$
\end{definition}
\begin{definition}\label{def:Stabilityconcept}
A closed set $\cA\subset \R^n$ is said to be \emph{uniformly globally asymptotically stable (UGAS)} 
for \eqref{eq:DiffIntroduction} if it is
\begin{itemize}[leftmargin=*]
\item \emph{uniformly globally stable (UGS)}, that is, there exists $\alpha\in \cK_\infty$ such that any solution $\phi\in\cS_{F,\cC}$ of \eqref{eq:DiffIntroduction} satisfies $|\phi(t)|_\cA\leq \alpha(|\phi(0)|_\cA)$ for all $t\in \dom \phi$;

\item\emph{uniformly globally attractive (UGA)}, that is, for each $\varepsilon >0$ and $r>0$ there exists $T>0$ such that for any solution $\phi\in\cS_{F,\cC}$ of \eqref{eq:DiffIntroduction}  with $|\phi(0)|_\cA\leq r$, if $t\in \dom \phi$ and $t\geq T$, then $|\phi(t)|_\cA\leq \varepsilon$.\hfill$\triangle$ 
  \end{itemize}
  \end{definition}
We underline that we do not require the forward completeness of maximal solutions. For that reason, the uniform global attractivity (UGA) introduced in Definition~\ref{def:Stabilityconcept} is sometimes called uniform global~\emph{pre-}attractivity (UGpA), for example in~\cite{GoebSanf12}, and UGAS is then called uniform global~\emph{pre-}asymptotic stability. Since we are not interested in completeness, we choose to avoid the prefix \emph{-pre}. For a thorough discussion about completeness property for constrained differential inclusions we refer to~\cite[Chapter 4]{AubCel84}. 

A common way to establish that a set $\cA$  is UGAS for system~\eqref{eq:DiffIntroduction} is through a positive definite real-valued function decreasing along the solutions. More precisely, given a closed set $\cA\subset \R^n$ we say that a function $V:\dom V\to \R$, with the domain $\dom V$ of $V$ open, is a \emph{smooth Lyapunov function} for system~\eqref{eq:DiffIntroduction} with respect to $\cA$ if $\cC\subset\dom V\subset \R^n$, $V\in \cC^1(\dom V,\R)$, there exist $\alpha_1,\alpha_2\in \cK_\infty$
such that 
\begin{equation} 
\label{V bounds} 
\alpha_1(|x|_\cA)\leq V(x)\leq \alpha_2(|x|_\cA), \quad \forall x\in \cC,
\end{equation} 
and there exists $\rho\in \cPd$ such that
\begin{equation} 
\label{eq:SmoothIneq}
\inp{\nabla V(x)}{f}\leq -\rho(|x|_\cA), \quad \forall x\in \cC,\forall f\in F(x). 
\end{equation}
It is standard that the existence of a smooth Lyapunov function implies that $\cA$ is UGAS; under some further assumptions on $\cC\subset \R^n$, $\cA\subset \R^n$ and $F:\cC\rightrightarrows \R^n$ the converse is also true, see \cite{TeelPral00b}, \cite{CLARKE199869}, \cite{Lin1996} for the case $\cC=\R^n$, or \cite[Theorem 7.31]{GoebSanf12} in the more general context of hybrid systems.
The main idea behind this paper is to relax the smoothness assumption on the Lyapunov function $V$ and to check the Lyapunov decrease inequality~\eqref{eq:SmoothIneq} only on subsets of the set $\cC$.

\subsection{Clarke Locally Lipschitz Lyapunov Functions}
Let us start by looking at functions that are not continuously differentiable but only locally Lipschitz continuous and by
recalling a standard Lyapunov condition for this case, \cite{ClarkeNONSmo}.  Consider a locally Lipschitz function $V:\dom V \to\R$, with $\dom V\subset\R^n$ open. The \emph{Clarke generalized gradient} at $x\in \dom V$ is the set
\begin{equation}\label{eq:ClarkeGradient}
\partial V(x):= \text{co} \left\{v\in \R^n\;{\Bigg \vert}\begin{aligned} \;  &\exists\, x_k \to x, \;x_k \notin \cN_V, \text{ s.t.} \\ &v=\lim_{k \to \infty} \nabla V(x_k)\end{aligned} \right\},
\end{equation}
where $\cN_V:=\{x\in \dom V \;\vert\;\nabla V(x)\;\text{does not exist}\}$ has zero Lebesgue measure, by Radamacher theorem, see \cite[Theorem~2.5.1]{clarke3}.

Consider a set $\cC\subset\R^n$, a set valued map  $F:\cC\rightrightarrows\R^n$ and a closed set $\cA\subset \R^n$. We say that a locally Lipschitz function $V:\dom V \to\R$,  with $\dom V$ open set satisfying $\dom V\supset \cC$, is a \emph{Clarke locally Lipschitz Lyapunov function} for system~\eqref{eq:DiffIntroduction} with respect to $\cA$ if
there exist $\alpha_1,\alpha_2\in \cK_\infty$ such that \eqref{V bounds} holds and there exists $\rho\in \cPd$ such that
\begin{equation}
\label{eq:ClarkeLyapunov}
\inp{v}{f}\leq -\rho(|x|_\cA), \quad \forall x\in \cC,\;\forall v\in \partial V (x),\;\forall f\in F(x).
\end{equation} 
In \cite[Section 4.5]{ClarkeNONSmo}, it is proven that the existence of a Clarke locally Lipschitz Lyapunov function implies UGAS of $\cA$. Moreover, when the inequality in~\eqref{eq:ClarkeLyapunov} is replaced by $\inp{v}{f}\leq 0$, then we say that $V$ is a Clarke locally Lipschitz \emph{weak} Lyapunov function. Similarly, the existence of a Clarke locally Lipschitz weak Lyapunov function implies UGS of $\cA$.  

We note that checking condition~\eqref{eq:ClarkeLyapunov} requires us to compute the Clarke generalized gradient of $V$ at \emph{each point} in $\cC$. We see here that this condition can be straightforwardly relaxed in some cases, as already noted in \cite{TeelPral00a} for continuous vector fields and differential equations.
\begin{proposition}\label{proposition:AlmEveCond}
Consider an open set $\cO\subset\R^n$, a locally Lipschitz function $V:\cO\to\R$,  an inner semicontinuous and locally bounded set valued map $F:\cO\rightrightarrows\R^n$, a closed set $\cA\subset \R^n$ and a continuous function $\gamma:\R_+\to\R$. If
\begin{equation}\label{eq:almostevery}
\inp{\nabla V(y)}{f}\leq \gamma(|y|_\cA),\;\;\;\forall f\in F(y),\;\forall y\in \cO\setminus\cN_V,
\end{equation}
it holds that
\begin{equation*}
\inp{v}{f}\leq\gamma(|x|_\cA)\;\;\;\forall x\in \cO,\,\forall v\in \partial V(x),\;\forall f\in F(x).
\end{equation*}
\end{proposition}
In particular, if $-\gamma\in \cPd$ in~\eqref{eq:almostevery}, then~\eqref{eq:ClarkeLyapunov} holds with $\cC=\cO$ and $\rho:=-\gamma$. But, we don't insist on $-\gamma\in \cPd$ in Proposition~\ref{proposition:AlmEveCond} to cover both $\gamma\equiv 0$ (certifying UGS) and $-\gamma\in \cPd$ (certifying UGAS) as well as a variety of more general decrease/increase conditions. The following proof is simply an adaptation, in the context of inner semicontinuous set-valued maps, of the reasoning presented in~\cite{TeelPral00b} for continuous functions $f:\R^n\to\R^n$.
\begin{proof}[of Proposition~\ref{proposition:AlmEveCond}]
Consider $x\in \cO$, and according to~\eqref{eq:ClarkeGradient}, take any sequence $x_k\to x$, $x_k\in \cO\setminus\cN_V$, such that $\lim_{k\to\infty}\nabla V(x_k)$ exists. Denote $v:=\lim_{k\to\infty}\nabla V(x_k)$ and pick any $f\in F(x)$. By inner semicontinuity and local boundedness of $F$, there exists a sequence $f_k\in F(x_k)$ such that $f_k\to f$. 
By equation~\eqref{eq:almostevery}, and by continuity of $\gamma$ and of the scalar product, we have
\[
\begin{aligned}
 \inp{\nabla V(x_k)}{f_k}&\leq \gamma(|x_k|_\cA),\\
\downarrow \;\;\;\;\;&\;\;\;\;\; \downarrow\\
\;\;\inp{v}{f}\;\;&\leq \gamma(|x|_\cA).
\end{aligned}
\]
The result follows from the arbitrariness of $v\in \partial V(x)$ and $f\in F(x)$.\qed
\end{proof}

Summarizing, when $F$ is inner semicontinuous and locally bounded, it suffices to check the Lyapunov inequality 
\emph{almost everywhere} in the open set $\cO$, that is at points where $\nabla V$ is defined and then the Clarke decrease condition~\eqref{eq:ClarkeLyapunov} holds everywhere in $\cO$. Combining~\cite[Section 4.5]{ClarkeNONSmo} and Proposition~\ref{proposition:AlmEveCond}, it is possible to guarantee UGAS (or UGS) using locally Lipschitz functions $V$ by only certifying the decrease at the points where $V$ is differentiable.

\subsection{Counterexamples: dense sets and non-inner semicontinuous maps}

Given an open set $\cC\subset \R^n$, note that a full measure subset of $\cC$, that is a set $\cS\subset \cC$ such that $\meas(\cC\setminus \cS)=0$, is always a \emph{dense} subset of $\cC$, in the sense that $\overline{\cS}$ contains $\cC$. The converse is not true in general: for example $\cS=\mathbb{Q}\subset \R$ is such that $\overline{\mathbb{Q}}=\R$ but $\meas(\mathbb{Q})=0$. 
One can ask: is it sufficient, for a general locally Lipschitz function, to check Lyapunov decrease inequalities only on a \emph{dense subset} of $\cC$? The answer is ``no'', as illustrated by the following example.

\begin{example}[Checking on a dense set]
The main idea of this example is taken from \cite{Rockafellar81}: 
Consider $\lambda\in (\frac{1}{2},1)$, and a measurable set $M\subset \R_{\geq 0}$ such that, for every non-empty interval $I\subset \R_{\geq 0}$, it holds that
\begin{equation}\label{eq:SplitInterval}
0<\meas(M\cap I)<\meas(I),
\end{equation}
(such sets are called interval-splitting) and, additionally,
\begin{equation}\label{eq:MBound}
\meas(M\cap[0,t])\geq \lambda t, \;\;\;\forall \;t>0.
\end{equation}
The construction of such a set is provided in~\cite[Lemma 2]{2019arXiv191013604D}, see also~\cite{Rudin1983} for the original construction of interval-splitting sets. 
\footnote{
In \cite{2019arXiv191013604D}, interval-splitting sets with the additional property \eqref{eq:MBound} are used to construct locally Lipschitz functions for which the steepest descent / subdifferential flow generated by the Clarke subdifferential has the pathological behavior of generating strictly increasing orbits.
}
Given a set $N\subset \R$, define the characteristic function of $N$ as
\[
\chi_N(s):=\begin{cases}
1,\;\;\;\text{if } s\in N,\\
0,\;\;\;\text{if } s\notin N.
\end{cases}
\]
Consider the function $W:\R_{\geq 0}\to\R$ defined as
\[
W(s):=\int_0^s \chi_M(r)-\chi_{M^c}(r)\, dr,
\]
where $M^c:=\R_{\geq 0}\setminus M$.
Using the same reasoning as in~\cite{Rockafellar81}, it can be proven that $W$ is locally Lipschitz, and~\eqref{eq:SplitInterval} implies that $\partial W(x)=[-1,1]$ for all $x\in \R_{\geq 0}$ and the sets
\[
\begin{aligned}
\cX_1&:=\{x\in \R_{\geq 0}\;\vert\; \nabla W(x)\;\text{exists}\; \wedge \;\nabla W(x)=1\},\\
\cX_{-1}&:=\{x\in \R_{\geq 0}\;\vert\; \nabla W(x)\;\text{exists}\; \wedge \;\nabla W(x)=-1\},
\end{aligned}
\]
are both dense subsets of $\R_{\geq 0}$. 
We prove next that~\eqref{eq:MBound} ensures that the function $W$ satisfies the bounds
\begin{equation}\label{eq:ExampleBounds}
(2\lambda-1)s\leq W(s)\leq s, \;\;\forall\;s\in \R_{\geq 0}.
\end{equation}
The upper bound is straightforward as
\[
W(s)=\meas(M\cap[0,s])-\meas(M^c\cap [0,s])\leq \meas([0,s])=s.
\]
The lower bound is obtained as follows from~\eqref{eq:MBound}:
\[
\begin{aligned}
W(s)&=\meas(M\cap[0,s])-\meas(M^c\cap [0,s])\\&\geq \lambda s-\meas([0,s]\setminus (M\cap[0,s]))\geq
\lambda s-(1-\lambda)s=(2\lambda -1)s.
\end{aligned}
\]
Consider now the differential equation
\[
\dot x=f(x)=x,
\]
and the candidate Lyapunov function $V:\R\to \R$ defined by $V(x):=W(|x|)$. The function $V:\R\to\R$ is locally Lipschitz and, by~\eqref{eq:ExampleBounds}, it is also positive definite and radially unbounded. 
Moreover, it holds that
\[
\nabla V(x)=\begin{cases}-1,\;\;&\forall\;x>0,\;x\in \cX_{-1},\\
+1,\;\;&\forall\; x<0,\;x\in -\cX_{-1}.
\end{cases}
\]
Consequently
\[
\inp{\nabla V(x)}{f(x)}=-|x|,\;\;\forall\,x\in \cS,
\]
where $\cS:=\cX_{-1}\cup-\cX_{-1}\setminus\{0\}$ is a \emph{dense} subset of $\R$ by construction.
In other words, $V$ is a positive definite and radially unbounded locally Lipschitz function for which the Lyapunov decrease inequality~\eqref{eq:SmoothIneq} is satisfied on a dense subset $\cS$ of $\R$.
On the other hand, the equilibrium point $0$ is clearly unstable.\hfill$\triangle$
\end{example} 

In the next section, we present a subclass of locally Lipschitz functions for which it is enough to check the Lyapunov decrease inequality on a dense subset of $\cC$.

Another question of interest in generalizing Proposition~\ref{proposition:AlmEveCond} is whether the inner semicontinuity hypothesis is in general necessary. First of all let us analyze how restrictive this assumption is, and when it is expected/ensured to hold. 
\begin{remark}[Inner Semicontinuity Assumption]
From Definition~\ref{def:Continuity}, all the continuity concepts somehow coincide for single-valued maps. More precisely, given $\cO\subset \R^n$ and $f:\cO\to \R^n$, let us consider the corresponding singleton-valued map  $F^f:\cO\rightrightarrows\R^n$ defined by  $F^f (x):=\{f(x)\}$. Then, given $x\in \cO$, the following conditions are equivalent:
\begin{itemize}
\item $f$ is continuous at $x$;
\item $F^f$ is outer semicontinuous and locally bounded at $x$;
\item $F^f$ is inner semicontinuous at $x$;
\end{itemize}
see for example~\cite[Chapter 1]{AubCel84}. Thus, in Proposition~\ref{proposition:AlmEveCond} and in what follows, when restricting the attention to differential \emph{equations} (or, equivalently, single-valued maps), the inner semicontinuity hypothesis coincides with the (quite usual) continuity condition on the right-hand side. 
With set-valued maps, a non-trivial example of inner (but not outer) semicontinuous map emerges in the following case.  Given a set $ U\subset \R^m$ and a function $f:\R^n\times U\to \R^n$, consider again the set-valued map $F^U:\R^n\to \R^n$ defined by $F^U(x)=\{f(x,u)\;\vert\;u\in U\}$ as in~\eqref{eq:COntrolDiffIncl}. If $f(\cdot, u):\R^n\to \R^n$ is continuous for all $u\in U$, then $F^U$ is inner semicontinuous (in $\R^n$) but in general it is \emph{not} outer semicontinuous, specifically, in the case where $U$ is not a compact subset of $\R^n$, see~\cite[Chapter 2, Proposition 1]{AubCel84}. 
On the other hand, differential inclusions arising from discontinuous differential equations as in~\eqref{eq:FilippovIntroduction} are outer semicontinuous but in general not inner semicontinuous, as we clarify in Example~\ref{example:flower}.\hfill$\triangle$
\end{remark}
In the following example we study the Filippov regularization of a discontinuous differential equation, and we show that the conditions of Proposition~\ref{proposition:AlmEveCond} are not sufficient in establishing UGAS, due to the lack of inner semicontinuity of the considered set-valued map.
\begin{example}[Violating Inner Semicontinuity]\label{example:flower}
Consider the differential inclusion~\eqref{eq:FilippovIntroduction},  with $\cC=\R^2$, and the set-valued map $F^{\text{Fil}}_f:\R^2 \rightrightarrows \R^2$ defined as the Filippov regularization~\cite{filippov1988differential} of the discontinuous linear system
\begin{equation}
\dot x=\begin{cases}
A_1 x,\;\;&\text{if}\;x^\top Q x\geq 0,\\
A_2 x,\;\;&\text{if}\;x^\top Q x< 0
\end{cases}
\end{equation}
where 
\[
A_1:=\begin{bmatrix}
-0.3 & -1 \\ 5 & -0.3
\end{bmatrix},\;A_2:=\begin{bmatrix}
-0.3 & 5 \\ -1 & -0.3
\end{bmatrix},\;Q:=\begin{bmatrix}
1 & 0\\0 & -1
\end{bmatrix}.
\]
By definition of Filippov regularization, defining $\cX_i:=\{x\in \R^2\,\vert\,(-1)^i x^\top Q x<0\}$, we have
\begin{equation}
F^{\text{Fil}}_f(x)=\begin{cases}
\{ A_1 x\}, \;&\text{if} \;x^\top Q x> 0,\;\text{equivalently}\;x\in \cX_1,\\
\{A_2 x\},\;&\text{if}\;x^\top Q x< 0,\;\text{equivalently}\;x\in \cX_2,\\
\co\{A_1 x,A_2 x\},\,\,\;&\text{if}\;x^\top Q x=0.
\end{cases}
\end{equation}
\begin{figure}
\begin{center}
\includegraphics[scale=0.75]{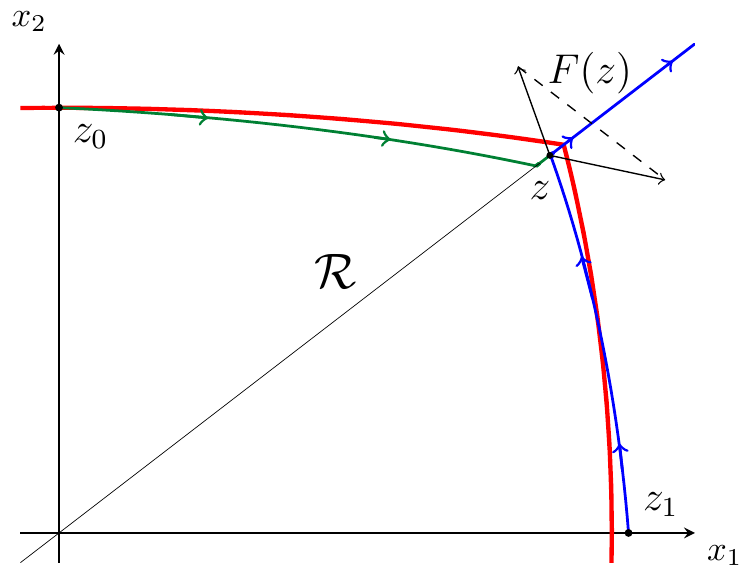}
\caption{Example \ref{example:flower}: In red line a level set of $V$ in \eqref{eq:lyapcounter}, in green and blue the trajectories of two solutions. The solution starting at $z_1$ enters the plotted sublevel set, reaches the point $z\in \mathcal{R}$ and then starts ``sliding'' towards infinity. }
\label{fig:flower}
\end{center}
\end{figure}
By construction~\cite{filippov1988differential}, $F$ is outer semicontinuous on $\R^2$, locally bounded, has compact and convex values, but it is \emph{not} inner semicontinuous at points $x\in \R^2\setminus\{0\}$ satisfying $x^\top Q x=0$.
Consider the locally Lipschitz function
\begin{equation}\label{eq:lyapcounter}
V(x):=\max\left\{x^\top P_1 x, x^\top P_2 x\right \},
\end{equation}
where $P_1:=\begin{bmatrix} 5 & 0 \\0 & 1\end{bmatrix}$ and $P_2:=\begin{bmatrix} 1 & 0 \\0 & 5\end{bmatrix}$.
It is easy to check that
\[
\begin{aligned}
&\bullet \;P_iA_i+A_i^\top P_i<0,\;\;\forall\;i\in \{1,2\},\\
&\bullet\;x^\top Q x> 0 \Leftrightarrow x^\top P_1 x>x^\top P_2 x,\\
&\bullet\;\bd(\cX_1)\cup\bd(\cX_2) =\{x\in \R^2\,\vert\,x^\top Q x=0\}\\
&\bullet\; \meas(\bd(\cX_i))=0,\;\; \text{for }i\in \{1,2\} .
\end{aligned}
\]
As a consequence we have
\[
\inp{\nabla V(x)}{f}<0;\;\forall x \in \R^2\setminus(\bd(\cX_1)\cup\bd(\cX_2)) ,\;\forall f\in F(x),
\]
which implies~\eqref{eq:almostevery}.
On the other hand, condition~\eqref{eq:ClarkeLyapunov} does not hold on $\cR:=\{x\in \R^2\;\vert\; x_1=x_2\}$, and one can see that every  solution of~\eqref{eq:FilippovIntroduction} starting at some $x_0\in \cR$, $x_0 \neq 0$ goes to infinity sliding along $\cR$. In particular, the origin is unstable. See Figure \ref{fig:flower} for a graphical representation.\hfill$\triangle$
\end{example}

\section{Main result}\label{sec:main}
 
Without any topological assumption on the set $\cC\subset\R^n$, we present here a class of locally Lipschitz Lyapunov functions associated to system~\eqref{eq:DiffIntroduction}, for which it suffices to check the Lyapunov inequality on a dense subsed of $\cC$.

\begin{definition}[The classes $\cL(\cC)$ and $\cLL(\cA,\cC)$]
\label{def:nicedefinition}
Consider $\cC\subset\R^n$. A function $V:\dom V\to\R$ (with $\dom V$ open), is a  {\em locally Lipschitz and locally finitely generated function on $\cC$} (and we write $V\in \cL(\cC)$) if $\overline{\cC}\subset\dom V$ and, \\for each $x\in \cC$,
\begin{itemize}
  \item[(a)]  There exists an open neighborhood $\cU(x)\subset \R^n$ such that there exists $L>0$ satisfying $|V(x')-V(x'')|\leq L|x'- x''|$, $\forall x',x''\in \cU(x)\cap \cC$ (i.e., $V$ is locally Lipschitz relative to $\cC$).
  \item[(b)] There exists a set $\cS(x)\subset \cC\cap \cU(x)$, such that $\nabla V(y)$ exists for all $y\in S(x)$, satisfying
\begin{equation}\label{eq:densityOfS}
\overline{\cS(x)}\supset \cC\cap \cU(x)\;\;  (\text{i.e. $\cS(x)$ is dense in $\cC\cap \cU(x)$}).
\end{equation}
\item[(c)] There exists a finite index set $\cI(x)$, and for each $i\in\cI(x)$, there are open sets $\cU_{i}\subset\R^n$ and $\cC^1$ functions $V_{i}:\cU_i\to\R$ such that each $y\in \cS(x)$ satisfies, for some $i\in \cI(x)$, 
 \begin{equation}\label{eq:PiecewiseCondDef}
y\in \cU_i,\;\;\;V(y)=V_{i}(y),\;\;\text{ and } \;\;\nabla V(y)=\nabla V_{i}(y).
 \end{equation}
\end{itemize} 
Considering a map $F:\cC\rightrightarrows \R^n$ and a closed set $\cA\subset\R^n$, we say that $V$ is a {\em locally Lipschitz and locally finitely generated strong Lyapunov  function for $\cA$ on $\cC$} ($V\in \cL_F(\cA,\cC)$) if $V\in \cL(\cC)$  and
\begin{itemize}[leftmargin=0.8cm]
 \item[(L1)] There exist  $\alpha_1,\alpha_2\in \cK_\infty$ such that 
 \begin{equation}
 \alpha_1(|x|_\cA)\leq V(x)\leq\alpha_2(|x|_\cA),  \;\;\;\forall\;x\in \cC.
\end{equation}
\item[(L2)] There exists $\rho\in\cPd$ such that, for each $x\in \cC$, each $y\in \cS(x)$ satisfies~\eqref{eq:PiecewiseCondDef} for some $i\in \cI(x)$ and moreover,
\begin{equation}\label{eq:LyapunovDecrease}
\inp{\nabla V_i(y)}{f}\leq -\rho(|y|_\cA), \;\;\forall \,f\in F(y).
\end{equation}
\end{itemize}
Finally, we say that $V$ is a locally Lipschitz and locally finitely generated  {\em weak} Lyapunov  function for $\cA$ on $\cC$ ($V\in \cL^{\circ}_F(\cA,\cC)$) if all the previous conditions hold with $\rho\equiv 0$ in~(L2).
\hfill$\triangle$
\end{definition} 

We now prove that if $V$ is in $\cLL(\cA,\cC)$ ($V$ is in $\cL^{\circ}_F(\cA,\cC)$, resp.) then the value of $V$ decreases (does not increase, resp.) along the solutions of system~\eqref{eq:DiffIntroduction}.

\begin{theorem} \label{theorem:MainTheo}
If $V\in \cLL(\cA, \cC)$ and $F:\cC\rightrightarrows \R^n$ is locally bounded and inner semicontinuous, then, for every  
$\phi\in\mathcal{S}_{F,\cC}$ and almost every $t\in \dom(\phi):=[0,T_\phi)$,
\begin{equation}\label{eq:flowingcond}
\frac{d}{dt} V(\phi(\cdot))(t)\leq-\rho(|\phi(t)|_A).
\end{equation}
Thus $\cA$ is UGAS for~\eqref{eq:DiffIntroduction}. Similarly, if $V\in \cLL^{\circ}(\cA, \cC)$, then~\eqref{eq:flowingcond} holds with $\rho\equiv0$ and thus $\cA$ is UGS for~\eqref{eq:DiffIntroduction}.
\end{theorem} 

\begin{proof} 
Consider any $\phi\in\mathcal{S}_{F,\cC}$. Since $\phi$ is absolutely continuous by definition, so is $t\mapsto V(\phi(t))$.
Then, $\phi$ and $V\circ \phi$ are differentiable almost everywhere, 
i.e., for almost every $t\in\dom\phi=[0,T_\phi)$,
$\dot{\phi}(t)$ and $\frac{d}{dt}V(\phi(\cdot))(t)$ exist, and for almost every such $t$,
$\dot{\phi}(t)\in F(\phi(t))$. 
Thus consider any such time $t\in [0, T_\phi)$, and denote $x:=\phi(t)\in \cC$ and $f:=\dot \phi(t) \in F(x)$. Consider the neighborhood $\cU(x)$ of $x$, the associated index set $\cI(x)$, the sets $\cU_{i}$,the functions $\{V_i\}_{i\in \cI(x)}$ and the set $\cS(x)$ given by Definition~\ref{def:nicedefinition}.
Let us now consider a sequence of times $t_k \searrow t$, and denote by $x_k := \phi(t_k)\in \cC$; by continuity of $\phi$,  $x_k\to x$.  Without loss of generality we can suppose $x_k \in \cU(x)\cap \cC$, $\forall k \in \N$ (possibly disregarding some initial points and relabeling). Let us note that, in general, it is possible that $x_k=\phi(t_k)\notin \cS(x)$, for some $k\in \N$. 
 Since, by Definition~\ref{def:nicedefinition}, $\overline{\cS(x)}\supset \cC\cap \cU(x)$, for any $k\in \N$ there exists an auxiliary sequence $x_{k,l}\in \cS(x)$ converging to $x_k$. By property~(c) of Definition~\ref{def:nicedefinition}, for each $l\in \N$, there exists $i_{k,l}\in \cI(x)$ satisfying $V(x_{k,l})=V_{i_{k,l}}(x_{k,l})$ and $\nabla V(x_{k,l})=\nabla V_{i_{k,l}}(x_{k,l})$. Since $\cI(x)$ is finite,we may consider, without relabeling, a subsequence of $x_{k,l}$,  such that $i_{k,l}=i_k$ for some $i_k\in I_x$, for every $l\in \N$. Similarly we may consider, without relabeling, a subsequence of $x_k$ such that  $i_k=i$, for some $i\in \cI(x)$ and every $k\in \N$. Then $V(x_{k,l})=V_i(x_{k,l})$, $\forall\,l\in \N$, and, by continuity of $V$ and $V_i$, the following holds
\begin{equation}\label{eq:ProofLimits}
\begin{aligned}
V(x_k)&=\lim_{l\to\infty}V(x_{k,l})=\lim_{l\to\infty}V_i(x_{k,l})=V_i(x_k),\\
V(x)&=\lim_{k\to\infty}V(x_k)=\lim_{k\to\infty}V_i(x_k)=V_i(x).
\end{aligned}
\end{equation}
 Hence, recalling that $V_i\in \cC^1(U_i,\R)$ and using identities~\eqref{eq:ProofLimits}, we have 
 \begin{equation}\label{eq:longequalities}
\begin{aligned}
\frac{d}{dt}V(\phi(\cdot))(t)&=\lim_{\tau \searrow t}\frac{V(\phi(\tau))-V(\phi(t))}{\tau-t}=\lim_{k\to \infty}\frac{V(x_k)-V(x)}{t_k-t}\\
&=\lim_{k\to \infty}\frac{V_i(x_k)-V_i(x)}{t_k-t}=\lim_{\tau \searrow t}\frac{V_i(\phi(\tau))-V_i(\phi(t))}{\tau-t}\\
&=\frac{d}{dt}V_i(\phi(\cdot))(t)=\inp{\nabla V_i(x)}{\dot \phi(t)}=\inp{\nabla V_i(x)}{f}.
\end{aligned}
\end{equation}
Now, for each $k\in \N$, we can choose a large enough $l=l_k$ so that $x_{k,l_k}\to x$ as $k\to\infty$. By inner semicontinuity  of $F$  there exists a sequence $f_k \in F(x_{k,l_{k}})$ such that $f_k \to f$ as $k \to \infty$. Finally by continuity of $\nabla V_i$, $\rho$ and the scalar product we have, as $k \to \infty$
\[
\begin{aligned}
\eqref{eq:LyapunovDecrease}\;\;\Rightarrow\;\; &\inp{\nabla V_i(x_{k,l_{k}})}{f_k}\leq -\rho(|x_{k,l_{k}}|_\mathcal{A}),\\
&\hskip1cm\downarrow \hskip2cm \downarrow\\
&\;\;\inp{\nabla V_i(x)}{f}\;\;\;\;\;\leq-\rho(|x|_\cA),
\end{aligned}
\]
and by~\eqref{eq:longequalities}, we can conclude that~\eqref{eq:flowingcond} holds. By a standard comparison argument,~\eqref{eq:flowingcond} implies that $\cA$ is UGAS for~\eqref{eq:DiffIntroduction}. The same argument could be used to infer that $V\in \cL^{\circ}_F(\cA,\cC)$ implies that~\eqref{eq:flowingcond} holds with $\rho\equiv 0$ and thus $\cA$ is UGS for~\eqref{eq:DiffIntroduction}.\qed
\end{proof}

\section{Comparisons with other classes of functions}\label{sec:Comparison}

In this section, we relate the class $\cL(\cC)$ introduced in Definition \ref{def:nicedefinition} to the class of piecewise $\cC^1$ functions introduced in \cite{Sch12}, in the case where $\cC\subset\R^n$ is an \emph{open set}. 
This  also allows us to investigate the relations between the class $\cLL(\cA,\cC)$ and the concept of Clarke locally Lipschitz Lyapunov functions defined in~\eqref{eq:ClarkeLyapunov}.

\subsection{Piecewise $\cC^1$ functions vs $\cL(\cC)$ functions}

In~\cite[Chapter 4]{Sch12}, the following class of functions is introduced.

\begin{definition}[Piecewise $\cC^1$ function; \cite{Sch12}]\label{def:piecwise}
Given an \emph{open} set $\cO\subset \R^n$, a \emph{continuous} function $V:\cO\to \R$ is called \emph{piecewise $\cC^1$ function on $\cO$} if for each $x\in \cO$ there exist an open neighborhood $\cU(x)\subset \cO$ of $x$, an index set $\cI(x)=\{1,\dots K\}$, a family $\cF=\{V_1, \dots V_K\}\subset \cC^1(\cU(x), \R)$ such that 
\begin{equation}
V(z)\in \{V_i(z)\;\vert\;i\in I_x\}, \;\;\forall z\in \cU(x).\tag*{$\triangle$}
\end{equation}
\end{definition}
Roughly speaking, a piecewise $\cC^1$ function is, locally, a continuous selection from (or patching of) a finite number of ``pre-defined" continuously differentiable functions. 
Properties of such functions include:
\begin{itemize}[leftmargin=*]
\item[(a)] If $V$ is piecewise $\cC^1$ on $\cO$ then it is locally Lipschitz continuous, see \cite[Prop.~4.1.2]{Sch12},
\item[(b)] There exists an open and dense subset $\cS\subset \cO$ where $V$ is continuously differentiable. Moreover, given any $x\in \cO$, the neighborhood $\cU(x)$ and the associated index set $\cI(x)$ in Definition~\ref{def:piecwise},  it follows from~\cite[Prop.~4.1.5]{Sch12} that for every $z\in \cU(x)\cap \cS$ there exists an $i\in \cI(x)$ such that $\nabla V(z)=\nabla V_i(z)$.
\item[(c)] Given any $x\in\cO$, the related neighborhood $\cU(x)$, and the index set $\cI(x)$ in Definition~\ref{def:piecwise}, define the following set of \emph{essentially active indexes} at $x\in \cO$:
\begin{equation}\label{eq:activeindexset}
\cI^{e}(x)=\left \{i\in \cI(x)\;\vert\;x\in \overline{\inn(\{y\in \cU(x)\;\vert\;V(y)=V_i(y)\})}\right\}.
\end{equation}
It follows by continuity that $V(x)=V_i(x)$ for all $i\in \cI^e(x)$, and in particular there exists a neighborhood $\cU'(x)$ of $x$, $\cU'(x)\subset \cU(x)$, such that
\[
V(z)\in \{V_i(z)\;\vert\;i\in \cI^e(x)\},\;\;\;\text{for all }z\in \cU'(x).
\]
\item[(d)] By \cite[Proposition~4.3.1]{Sch12}, for any $x\in \cO$ 
\begin{equation}\label{eq:ClarkeGradPiece}
\partial V(x)=\co\{\nabla V_i(x)\;\vert\;i\in \cI^e(x)\}.
\end{equation}
\end{itemize}
Properties (a) and (b) above imply that a piecewise $\cC^1$ function satisfies properties (a), (b) and (c) of Definition~\ref{def:nicedefinition}, and thus a piecewise $\cC^1$ function is an $\cL(\cO)$ function for
open $\cO$. The converse is also true, as proven below, and thus we have the following result:

\begin{proposition}\label{lemma:equivalence}
Consider an \emph{open} set $\cO\subset \R^n$. A function $V:\dom(V)\to\R$ is in $\cL(\cO)$ if and only if it is piecewise $\cC^1$ on $\cO$.
\end{proposition}

\begin{proof}
Let $V\in \cL(\cO)$. We will show that it is piecewise $\cC^1$ on $\cO$. The continuity of $V$ on $\cO$ is trivial, since  $V$ is locally Lipschitz relative to $\cO$, by item (a) of Definition~\ref{def:nicedefinition}. Take any $x\in \cO$ and consider the neighborhood $\cU(x)$, the set $\cS(x)$, the index set $\cI(x)$, the sets $\cU_i$, and the functions $V_i:\cU_i\to \R$ given in Definition~\ref{def:nicedefinition} of $\cL(\cO)$. Since $\cO$ is open, without loss of generality, we suppose that $\cU(x)\subset \cO$. We firstly prove that 
\begin{equation}\label{eq:ViDef}
\cU(x)\subset \bigcup_{i\in \cI(x)}\cV_i:=\bigcup_{i\in \cI(x)}\{z\in \cU_i\;\vert\;V(z)=V_i(z)\}.
\end{equation}
To prove~\eqref{eq:ViDef}, take any $z\in \cU(x)$. If $z\in \cS(x)$, then $V(z)=V_i(z)$ for some $i\in \cI(x)$ and thus $z\in \bigcup_{i\in I_x} \cV_i$. 
If $z\not\in \cS(x)$, then from conditions~(b) and (c) in Definition~\ref{def:nicedefinition} there exist $z_k\in \cS(x)$, $z_k\to z$ such that $V(z_k)=V_{i_k}(z_k)$ for some $i_k\in \cI(x)$. 
By finiteness of $\cI(x)$, without loss of generality we can assume $i_k=i$ for some $i\in \cI(x)$ and for all $k\in \N$. Then by continuity of $V$ and $V_i$, we have
\[
V(z)=\lim_{k\to \infty}V(z_k)=\lim_{k\to \infty}V_i(z_k)=V_i(z),
\]
 which shows that $z\in \cV_i\subset \bigcup_{i\in \cI(x)} \cV_i$.  
Concluding, define $\cI'(x):=\{i\in \cI(x)\;\vert\;x\in \cU_i\}$ and $\cU'(x)\subset \cU(x)$ as $\cU'(x):=\bigcap_{i\in \cI'(x)} \cU_i\cap \cU(x)$. By~\eqref{eq:ViDef}, $\cI'(x)$ is not empty and thus $\cU'(x)$ is an (open) neighborhood of $x$. Again by~\eqref{eq:ViDef} we have
 \[
V(z)\in \{V_i(z)\;\vert\;i\in \cI'(x)\},\;\;\forall\,z\in \cU'(x),
 \]
 concluding the proof.
\qed
\end{proof}
Definition~\ref{def:nicedefinition} and the class $\cL(\cC)$ are intended to be used on sets $\cC$ that are not open. In particular, in later sections they are used on the flow set of a hybrid system. 
It is less clear, and outside the scope of this paper, if and how can Definition~\ref{def:piecwise} be generalized to describe piecewise $\cC^1$ structure on a set that is not open, and then if and how such a  generalization relates to the class from Definition~\ref{def:nicedefinition}.

As an illustration of Definition~\ref{def:piecwise}, we will show in Section~\ref{sec:discFurth} how certain classes of Lyapunov functions (such as max-of-quadratics introduced in~\cite{goebel2}) relate to piecewise $\cC^1$ functions. Moreover, when the set $\cC$ is not necessarily open, we underline differences between the function class $\cL(\cC)$ and the class of piecewise $\cC^1$ functions on a neighborhood of $\cC$ with the help of Example~\ref{ex:CircleExample} (which appears in Section~\ref{sec:SubsecClarkevsBlabla}).

\subsection{Clarke Lyapunov Functions vs $ \cLL(\cA,\cC)$ functions}
\label{sec:SubsecClarkevsBlabla}
We study here the relation between \emph{locally Lipschitz Clarke Lyapunov functions} (in the sense of~\eqref{eq:ClarkeLyapunov}) and locally finitely generated Lyapunov functions $\cLL(\cA,\cC)$ (in the sense of Definition~\ref{def:nicedefinition}). The next lemma shows that a function $V\in \cLL(\cA,\cC)$ also satisfies the condition~\eqref{eq:ClarkeLyapunov} in the interior of $\cC$, and thus is a Clarke locally Lipschitz Lyapunov function  if $\cC$ is open. 

\begin{lemma}
Given a closed set $\cA\subset \R^n$, a set $\cC\subset \R^n$ and  $F:\cC\rightrightarrows \R^n$ a locally bounded and inner semicontinuous set-valued map. Consider a function $V\in \cLL(\cA,\cC)$, then for every $x\in \inn(\cC)$, we have 
\begin{equation}\label{eq:ClarkeCondition2}
\inp{v}{f}\leq -\rho(|x|_{\cA}),\;\;\forall v\in \partial V(x),\;\;\forall f\in F(x).
\end{equation}
If $V\in \cLL^\circ(\cA,\cC)$ then, for every $x\in \inn(\cC)$, inequality~\eqref{eq:ClarkeCondition2} holds with $\rho\equiv0$. 
\end{lemma}
\begin{proof}
Consider any $x\in \inn(\cC)$, and the open neighborhood $\cU(x)$ given by Definition~\ref{def:nicedefinition}; we can suppose, without loss of generality, that $\cU(x)\subset \inn(\cC)$. Thanks to Proposition~\ref{lemma:equivalence}, $V$ is piecewise $\cC^1$ on $\cU(x)$, and thus, recalling~\eqref{eq:ClarkeGradPiece}, we have
\[
\partial V(x)=\co\{\nabla V_i(x)\;\vert\;i\in \cI^e(x)\}.
\]
The proof is carried out by showing
\begin{equation}\label{eq:inLemmmaClarke}
\inp{\nabla V_i(x)}{f}\leq -\rho(|x|_\cA),\;\;\forall\,i\in \cI^e(x),\;\forall\,f\in F(x).
\end{equation}
For proving~\eqref{eq:inLemmmaClarke}, consider any $i\in \cI^e(x)$. By definition~\eqref{eq:activeindexset}, there exists a sequence $x_k\to x$ such that $x_k\in \inn(\cV_i)$ for all $k\in \N$, where $\cV_i:=\{z\in \cU(x)\;\vert\;V(z)=V_i(z)\}$. By density of $\cS(x)$,
given in~\eqref{eq:densityOfS} in Definition~\ref{def:nicedefinition}, for each $k\in \N$ there exists an auxiliary sequence $x_{k,l}\to x_k$ as $l\to \infty$, such that $x_{k,l}\in \cS(x)\cap \inn(\cV_i)$, for all $l\in \N$. For each $k\in \N$, we choose a large enough $l=l_k$ so that $x_{k,l_k}\to x$ as $k\to\infty$. By construction $x_{k,l_k}\in \cS(x)\cap \inn(\cV_i)$, for all $k\in \N$. Consider any $f\in F(x)$, by inner semicontinuity of $F$ we can find a sequence $f_k\in F(x_{k,l_k})$ such that $f_k\to f$ as $k\to \infty$. By definition of $\cV_i$ and equation~\eqref{eq:LyapunovDecrease}, we have
\[
\inp{\nabla V_i(x_{k,l_k})}{f_k}\leq -\rho(|x_{k,l_k}|_\cA),\;\;\;\forall k\in\N,
\]
and by continuity of $\nabla V_i$ and $\rho$ this implies
\[
\inp{\nabla V_i(x)}{f}\leq-\rho(|x|_\cA).
\]
Since $x\in \inn(\cC)$, $i\in \cI^e(x)$ and $f\in F(x)$ are arbitrary, we can conclude the proof. The same argument can be used in the case $V\in \cLL^\circ(\cA,\cC)$ with $\rho\equiv 0$. \qed
\end{proof}

The lemma showed, roughly, that a function $V\in \cLL(\cA,\cC)$ is a Clarke locally Lipschitz Lyapunov function if $\cC$ is open. It is not clear how, under the conditions of Definition~\ref{def:nicedefinition}, to consider the condition 
\eqref{eq:ClarkeCondition2} at points on the boundary $\bd(\cC)$. The next example shows how a function in $\cLL(\cA,\cC)$ 
can work when $\cC$ has no interior. 

\setcounter{example}{2}
\begin{example}\label{ex:CircleExample}
Consider the closed set 
\[
\cC:=\{x:=(x_1,x_2)^\top\in \R^2 \;\vert\; (x_1-1)^2+(x_2-1)^2=2\;\wedge x_2\geq 0\},
\]
 represented by the red line in Figure~\ref{fig:ExampleCircle}, and the vector field $f:\cC\to \R^2$, defining system
 \begin{equation}\label{eq:Examplesystem}
\dot x=f(x)=|x|(Ax+b),\;\;\;\;\;x\in\cC,
\end{equation}
where $A=\begin{pmatrix}
0&-1\\1&0
\end{pmatrix}$ and $b=(1,-1)^\top$.\\
Note that $f(0)=0$ and $f(x)\neq 0$ for all $x\in \cC\setminus \{0\}$ (since $A^{-1}b=(1,1)^\top\notin \cC$). Denoting by $T_\cC(x)$ the tangent cone of $\cC$ at $x$ (see e.g. \cite[Definition 6.1]{rockafellar}), we have $f(x)\in T_\cC(x)$, for every $x\in \cC$, see Figure~\ref{fig:ExampleCircle} for a graphical representation.\\
We want to prove that $\cA=\{0\}$ is UGAS for system~\eqref{eq:Examplesystem}, constructing a function $V\in \cLL(\{0\},\cC)$.
To this end, consider three functions $V_i\in \cC^1(\R^2,\R)$, $i \in \{1,2,3\}$ given by
\begin{equation}\label{eq:ExampleFunction}
V_1(x):=x_2,\;\;V_2(x):=x_1+2,\;\;V_3(x):=-x_2+6.
\end{equation}
We show below that the function $V:\R^2\to \R$, given below, satisfies $V\in\cLL(\{0\},\cC)$. 
\[
V(x):=\begin{cases}
V_1(x)\;\;\text{if } x_1\leq 0,\\
V_2(x)\;\;\text{if } 0<x_1<2,\\
V_3(x)\;\;\text{if } x_1\geq 2.
\end{cases}
\]

\begin{figure}
  \centering
  \begin{tabular}{lcc}
    \includegraphics[width=.45\linewidth,height=170pt]{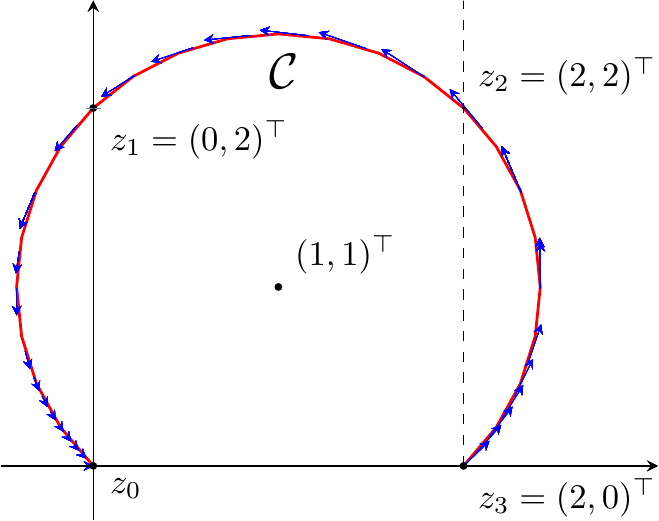} &
    \includegraphics[width=.45\linewidth,height=170pt]{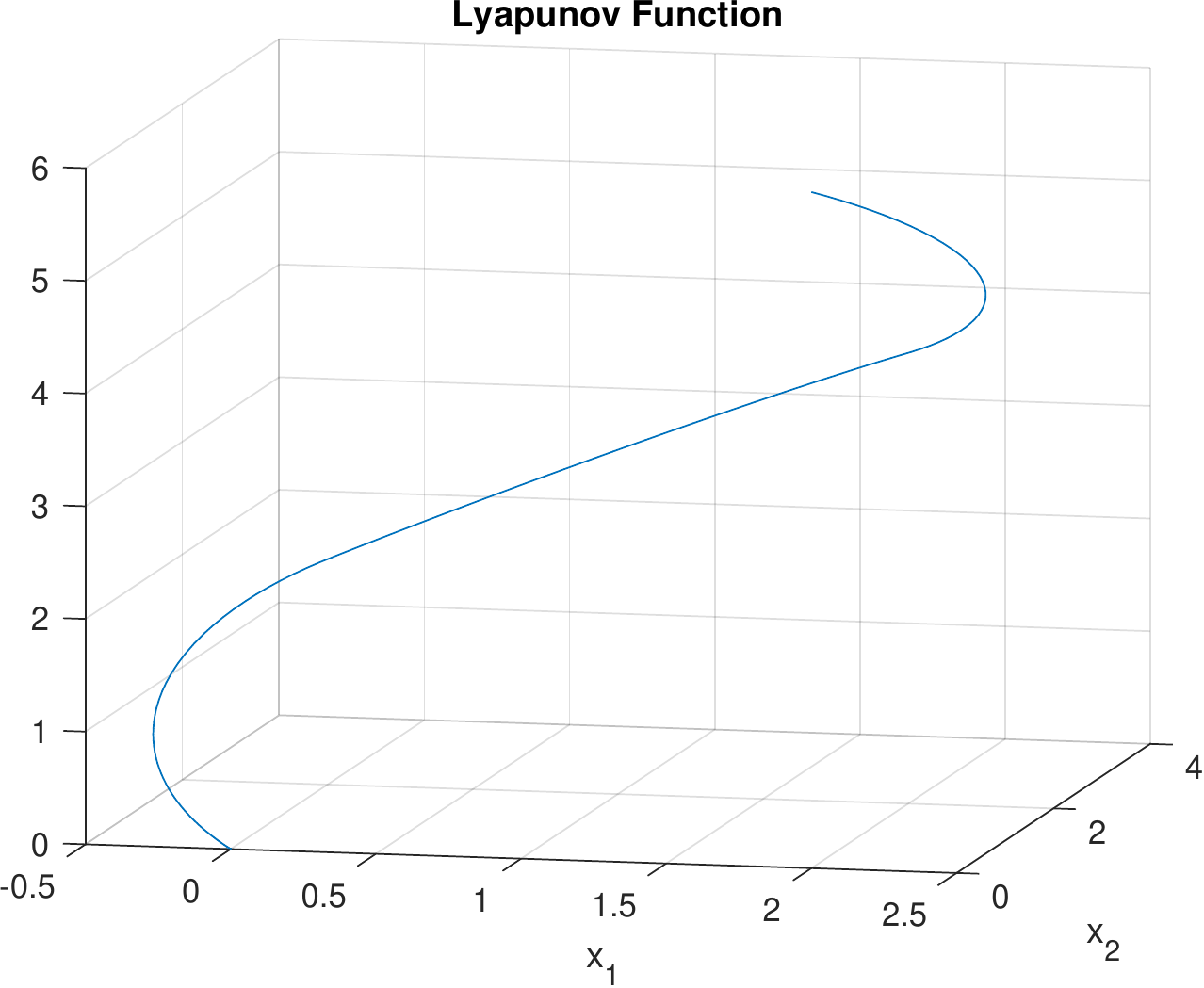}
  \end{tabular}
  \caption{Example~\ref{ex:CircleExample}, the system~\eqref{eq:Examplesystem} and the behavior of the function $V$.}\label{fig:ExampleCircle}
\end{figure}

\emph{Step 1: Local Lipschitz continuity and positive definiteness in $\cC$:}
By definition, $V$  is continuously differentiable in $\cC\setminus \{z_0,z_1,z_2,z_3\}$, where $z_0=(0,0)^\top$, $z_1:=(0,2)^\top$, $z_2:=(2,2)^\top$ and $z_3:=(2,0)^\top$, as shown in Figure~\ref{fig:ExampleCircle}. At these points $V$ is continuous because 
\[
\begin{aligned}
V(z_1)=V_1(z_1)&=V_2(z_1)=2,\\
V(z_2)=V_2(z_2)&=V_3(z_2)=4,\\
V(z)&=V_1(z),\;\;\;\;\;\;\;\;\;\;\;\forall z\in \ball(z_0,1)\cap\cC.\\
V(z)&=V_3(z),\;\;\;\;\;\;\;\;\;\;\;\forall z\in \ball(z_3,1)\cap\cC.
\end{aligned}
\]
It is also easily checked that $V$ is locally Lipschitz relative to $\cC$ with constant $L(x)\equiv 1$.
Finally, from a case-by-case analysis, it is easily checked that
$
\frac{|x|}{2}\leq V(x)\leq 3|x|$, $\forall x\in \cC$.

\emph{Step 2: Lyapunov conditions:} By definition of $V$ and $f$ we have
\[
\nabla V(x)=\begin{cases}

(0,1)^\top\;\;&\text{if }x_1<0\\\
(1,0)^\top\;\;&\text{if }0<x_1<2\\
(0,-1)^\top\;\;&\text{if }x_1>2\\
\end{cases}
\]
and computing we obtain
\[
\inp{\nabla V(x)}{f(x)}=\begin{cases}
|x|(x_1-1)\;\;&\text{if }x_1<0,\\
|x|(-x_2-1)\;\;&\text{if }0<x_1<2,\\
|x|(1-x_1)\;\;&\text{if }x_1>2.
\end{cases}
\]
Concluding, we may prove the almost everywhere condition 
\begin{equation}\label{eq:LyapunovExample}
\inp{\nabla V(x)}{f(x)}\leq- |x|, \;\;\forall x\in \cC\setminus \{z_1,z_2,z_3\}.
\end{equation}
It is thus clear that conditions (b), (c) and (L2) of Definition~\ref{def:nicedefinition} hold, by choosing $\cU(z_i)=\ball(z_i,1)$, $\cI(z_i)=\{i,i+1\}$, $\cS(z_i)=\cU(z_i)\setminus\{z_i\}$, for $i\in \{1,2\}$. For $z_3$ it suffices to choose  $\cU(z_3)=\ball(z_3,1)$, $\cI(z_3)=\{3\}$, $\cS(z_3)=\cU(z_3)\setminus\{z_3\}$. For the other points in $\cC$ the claim is trivial from~\eqref{eq:LyapunovExample}. We have thus proved that $V\in \cLL(\{0\},\cC)$, and Theorem~\ref{theorem:MainTheo} implies UGAS of the origin for system~\eqref{eq:Examplesystem}. See Figure~\ref{fig:ExampleCircle} on the right for a graphical representation of the function $V$.

\emph{Comparison with Clarke Lyapunov functions: }
It is easily seen that the function $V$ fails to be continuous outside $\cC$; more precisely, for every open set $\cO$ such that $\cC\subset \cO$, $V$ is not continuous on $\cO$, because discontinuity is inevitable in any neighborhood of $z_1$ and $z_2$. In particular, for any open set $\cO\supset \cC$, the function $V$ is not piecewise $\cC^1$ on $\cO$, recall Definition~\ref{def:piecwise}. Moreover, it is clear that the Clarke gradient cannot be defined at $z_1$ and $z_2$, and thus $V$ is not a Clarke Lyapunov function, as defined in~\eqref{eq:ClarkeLyapunov}. We emphasize that alternative constructions are possible but what is appealing about our design is its intuitive nature.\hfill$\triangle$
\end{example}

\section{Global piecewise structure}\label{sec:discFurth}

The properties in Definitions~\ref{def:nicedefinition} and~\ref{def:piecwise} are, in a sense, local, because they require that each 
$x\in \dom V$ has neighborhood on which $V$ has piecewise structure and is built from finitely many smooth functions, but this structure and the smooth functions may be different for different $x$. Below, we introduce a family of locally Lipschitz functions for which the  piecewise structure is global, in the sense that $V$ is globally obtained by ``gluing together'' a finite number of smooth functions.

\begin{definition}[Proper Piecewise $\mathcal{\cC}^1$ functions]\label{def:patchy}
Let $\cO\subset\R^n$ be an open set. A continuous function $V:\cO \to \R$ is called a \emph{proper piecewise $\cC^1$ function} on $\cO$ if there exist $\cI=\{1, \dots, K\}$, closed sets  $\{\cX_i\}_{i\in \cI}$, open sets $\{\cO_i\}_{i\in \cI}$  and continuously differentiable functions $V_i:\cO_i \to \R$, $i \in \cI$, such that:
\begin{enumerate}[leftmargin=*, label=(\Alph*)]
\item $\cX_i\cap \cO\subset \cO_i$, for all $i \in \cI$,\label{enum:a}
\item $\overline{\inn(\cX_i)}=\cX_i$, (namely $\cX_i$ is \emph{regular-closed}), for all $i \in \cI$,\label{enum:b}
\item $\cO \subset\bigcup_{i\in \cI}\cX_i$,\label{enum:c}
\item $V(x)=V_i(x), \;\;\;\text{if} \;x \in \cX_i$.\label{enum:d} \hfill $\triangle$
\end{enumerate}
\end{definition}

\begin{remark}
If $V:\cO\to \R$ is a proper piecewise $\cC^1$ function, then it is  piecewise $\cC^1$, as in Definition~\ref{def:piecwise}. Indeed, it suffices to select $\cU(x)\equiv \cO$, $\cI(x)\equiv \cI$ and $\cF\equiv\{V_1,\dots, V_K\}$.\hfill$\triangle$
\end{remark}

For this class of proper piecewise $\cC^1$ functions, given a set valued map $F:\cC\rightrightarrows \R^n$, in the case where $\cC$ is regular-closed (that is $\cC=\overline{\inn(\cC)}$), it may be computationally easier to check the conditions of Definition~\ref{def:nicedefinition}, as the next result suggests. The sufficient conditions below may hold when conditions based on the Clarke subgradient don't; this is illustrated in Example \ref{example:clegg}.

 \begin{theorem}\label{proposition:piecwiseprop}
 Consider a set $\cC\subset \R^n$ such that $\overline{\inn(\cC)}=\cC$ and let  $V :\cO\to \R$ be a proper piecewise $\cC^1$ function with $\cC\subset \cO$. Consider a closed set $\cA\subset \R^n$ and a locally bounded and inner semicontinuous set-valued map $F:\cC \rightrightarrows \R^n$. 
If there exist $\alpha_1,\alpha_2\in \cK_\infty$, and $\rho \in \cPd$ such that
\begin{equation}\label{eq:piecwisepositive}
\alpha_1(|x|_\cA)\leq V(x)\leq\alpha_2(|x|_\cA),\;\;\forall x\in \cC,
\end{equation}
\begin{equation}\label{eq:piecewiseFlowcond}
\inp{\nabla V_i(x)}{f}\leq -\rho(|x|_\cA),\;\;\;\;\begin{cases} \forall x\in \inn(\cX_i)\cap \inn(\cC),\\ 
\forall f\in F(x),\,\forall i\in \cI,\end{cases}
\end{equation}
then $V\in \cLL(\cA,\cC)$. Thus $\cA$ is UGAS
 for system~\eqref{eq:DiffIntroduction}. Considering $\rho\equiv 0$ in~\eqref{eq:piecewiseFlowcond}, then $V\in \cLL^\circ(\cA,\cC)$ and $\cA$ is UGS
 for system~\eqref{eq:DiffIntroduction}.
  \end{theorem}
 \begin{proof}
 We prove the theorem by showing that all the hypotheses of Definition~\ref{def:nicedefinition} hold.
The facts that $\cC \subset \cO$ and that $V$ is locally Lipschitz relative to $\cC$ are straightforward. Condition \eqref{eq:piecwisepositive} is exactly (L1) of Definition~\ref{def:nicedefinition}.
It only remains to prove (b),(c) and (L2) of Definition~\ref{def:nicedefinition}. For any $x\in \cC$, we consider $\cU(x)\equiv\cO$ and define $\cS(x)\equiv\cS:=\bigcup_{i\in \cI} \inn(\cX_i)\cap \inn(\cC)$. We now prove that $\overline{\cS}\supset \cC$. Consider a point $x\in \cC$,  recalling that $\overline{\inn(\cC)}=\cC$ we consider a sequence $x_k\in \inn(\cC)$ such that $x_k\to x$, as $k\to \infty$. By properties~\ref{enum:b} and~\ref{enum:c} of Definition~\ref{def:patchy}, we have
\[
\overline{\bigcup_{i\in I} \inn(\cX_i)}=\bigcup_{i\in I} \overline{\inn(\cX_i)}=\bigcup_{i\in I} \cX_i\supset\cO\supset\cC\supset\inn(\cC).
\]
Thus, for each $k\in \N$, there exists an auxiliary sequence $x_{k,l}\in \bigcup_{i\in I} \inn(\cX_i)\cap \inn(\cC)$ such that $x_{k,l}\to x_k$ as $l\to \infty$.  For each $k\in \N$, we choose a large enough $l=l_k$ so that $x_{k,l_k}\to x$ as $k\to\infty$. By construction, $x_{k,l_k}\in \bigcup_{i\in\cI}\inn(\cX_i)\cap \inn(\cC)$, proving the density of $\cS$ in $\cC$ (and thus (b) of Definition~\ref{def:nicedefinition}).
Now, for each $y\in \cS$ there exists $i\in \cI$ such that $y\in \inn(\cX_i)$ and since by condition~\ref{enum:d}, $V$ coincides with $V_i$ in the open set $\inn(\cX_i)$, we have $\nabla V(y)=\nabla V_i(y)$ proving condition (c) of Definition~\ref{def:nicedefinition}. Finally, by condition \eqref{eq:piecewiseFlowcond} we obtain (L2) of Definition~\ref{def:nicedefinition}, concluding the proof that $V\in \cLL(\cA,\cC)$. Then UGAS follows from Theorem~\ref{theorem:MainTheo}. The case with $\rho\equiv 0$ is completely analogous. \qed
 \end{proof}
We prove that the family of proper piecewise $\cC^1$ functions is closed under the pointwise maximum and pointwise minimum operators. 

\begin{proposition}\label{lemma:Maxmin}
Consider an open set $\cO\subset\R^n$. The class of proper piecewise $\cC^1$ functions on $\cO$ is closed under the operations of pointwise maximum and pointwise minimum of finitely many functions. More precisely, given  $V_1,\dots, V_K:\cO\to \R$ proper piecewise $\cC^1$ functions on $\cO$, the functions $V_{M}, V_m:\cO\to \R$ defined by
\[
\begin{aligned}
V_M(x)&:=\max_{i=1,\dots K}\{V_i(x)\},\\
V_m(x)&:=\min_{i=1,\dots K}\{V_i(x)\},
\end{aligned}
\;\;\;\;\;\;\forall \;x\in \cO,
\]
are proper piecewise $\cC^1$ on $\cO$.
\end{proposition}
\begin{proof}[Sketch of the Proof]
We note that if $V$ is proper piecewise $\cC^1$ so is $W:=-V$. Moreover, by definition of pointwise maximum and pointwise minimum operators, $\max\{a,b,c\}=\max\{a,\max\{b,c\}\}$ and $\min\{a,b\}=-\max\{-a,-b\}$ for all $a,b,c\in \R$. Therefore, it suffices to prove that, given $V_1,V_2:\cO\to \R$ proper piecewise $\cC^1$ functions, the max function $V_M:\cO\to \R$ defined by
\[
V_M(x):=\max\{V_1(x),V_2(x)\},\;\;\forall \;x\in \cO,
\]
is proper piecewise $\cC^1$ on $\cO$, since the general statement follows iterating this property.
Toward this end, let us consider $V_1,V_2:\cO\to\R$ proper piecewise $\cC^1$ functions.
For both $j=1,2$, we can consider $\cI^j=\{1,\dots, K_j\}$, closed sets $\{\cX^j_i\}_{i\in \cI^j}$, open sets $\{\cO^j_i\}_{i\in \cI^j}$  and continuously differentiable functions $V^j_i:\cO^j_i \to \R$, $i \in \cI^j$, such that Definition~\ref{def:patchy} is satisfied for $V_j$, $j=1,2$.
Define the open sets
\[
\begin{aligned}
\cY_0&:=\inn(\{x\in \cO\;\vert\;V_1(x)=V_2(x)\}),\\
\cY_1&:=\{x\in \cO\;\vert\;V_1(x)>V_2(x)\},\\
\cY_2&:=\{x\in \cO\;\vert\;V_1(x)<V_2(x)\}.
\end{aligned}
\]
Let us take $\cI^M:=\{1,\dots, K_1,K_1+1,\dots,2K_1,2K_1+1,\dots 2K_1+K_2\}$ and define 
\[
\cX_i^M:=\begin{cases}
\overline{\cX^1_i\cap \cY_0},\;\;\;\;&\text{if }\;1\leq i\leq K_1,\\
\overline{\cX^1_{i-K_1}\cap \cY_1},\;\;\;\;&\text{if }\;K_1+1\leq i\leq 2K_1,\\
\overline{\cX^2_{i-2K_1}\cap \cY_2},\;\;\;\;&\text{if }2K_1+1\leq i\leq 2K_1+K_2,
\end{cases}
\]
and
\[
(\cO_i^M,V^M_i):=\begin{cases}
(\cO^1_i, V^1_i)\;\;\;\;&\text{if }\;1\leq i\leq K_1,\\
(\cO^1_{i-K_1},V^1_{i-K_1})\;\;\;\;&\text{if }\;K_1+1\leq i\leq 2K_1,\\
(\cO^2_{i-2K_1},V^2_{i-2K_1})\;\;\;\;&\text{if }2K_1+1\leq i\leq 2K_1+K_2.
\end{cases}
\]
It can be shown that $V_M$, $\cI^M$, $\{\cX_i^M\}_{i\in \cI^M}$, $\{\cO_i^M\}_{i\in \cI^M}$ and $\{V^M_i\}_{i\in \cI^M}$ satisfy Definition~\ref{def:patchy}. Conditions~\ref{enum:a} and~\ref{enum:d} are straightforward; Condition~\ref{enum:b} follows from the fact that given $\cX\subset \R^n$ such that $\cX=\overline{\inn(\cX)}$ and an open set $\cU\subset \R^n$, one has
\[
\overline{\inn({\overline{\cX\cap \cU}})}=\overline{\cX\cap \cU}.
\]
Condition~\ref{enum:c} holds recalling that $\cup_{i\in \cI^1}\inn(\cX_i^1)$ and $\cup_{i\in \cI^2}\inn(\cX_i^2)$ are dense in $\cO$ and  $\cO\subset \cY_0\cup\overline{\cY_1}\cup\overline{\cY_2}$.\qed
\end{proof}
\begin{remark}
Let $\cO\subset\R^n$ be an open set, and consider $V_1,\dots, V_K:\cO\to \R$ proper piecewise $\cC^1$ functions. We consider a \emph{max-min} function $V_{\Mm}:\cO\to \R$ defined by 
\begin{align}\label{eq:Mmm}
V_{\Mm}(x)&:=\max_{j \in \{1, \dots, J\}} \left \{\min_{k \in S_j} \{V_k(x)\} \right \},\;\;\forall\,x\in \cO
\end{align}
where $J\geq1$ and $S_1, \dots, S_J\subset \{1, \dots, K\}$ are non-empty subsets, see also \cite{dellarossa19} and references therein for a thorough discussion about this family of functions. Iterating the result in Proposition~\ref{lemma:Maxmin}, it trivially holds that $V_{\Mm}$ is proper piecewise $\cC^1$ on $\cO$. Remarkably, any piecewise \emph{affine} function (PWA) can be represented in the form~\eqref{eq:Mmm} with $V_1,\dots V_K:\R^n\to \R$ affine functions, that is the so-called lattice representation, see~\cite{XuVanDer10} and references therein.\hfill$\triangle$
\end{remark}

\section{Application to hybrid dynamical systems}\label{sec:HybridSystems}

A broad framework where constrained differential inclusions~\eqref{eq:DiffIntroduction} appear is \emph{hybrid dynamical systems},~\cite{GoebSanf12}. Given $\cC,\cD\subset \R^n$, $F:\dom F\rightrightarrows \R^n$, $G:\dom G\rightrightarrows \R^n$, such that $\cC\subset\dom F$ and $\cD \subset \dom G$, a hybrid dynamical system $\mathcal{H}=(\cC,\cD,G,F)$ is
\begin{equation}\label{eq:hybsys}
\mathcal{H}:\;
\begin{cases}
\dot x\in F(x),\;\;&\;x\in \cC,\\
x^+\in G(x),\;\;&\;x\in \cD.
\end{cases}
\end{equation}
For all relevant definitions, of a solution to \eqref{eq:hybsys} and of stability concepts invoked below, 
we refer to \cite{GoebSanf12}.

\subsection{Stability Conditions using $\cLL(\cA,\cC)$}

An extension of Theorem~\ref{theorem:MainTheo} to hybrid systems is as follows. 

\begin{theorem}\label{proposition:StabHybrid}
Given hybrid system~\eqref{eq:hybsys}, suppose that $F:\dom F\rightrightarrows \R^n$ is locally bounded and inner semicontinuous with $\cC\subset \dom F$. Given a closed set $\cA$, suppose that
$V:\dom V\to\R$ is such that 
\begin{itemize} 
  \item[(a)] $V\in \cLL(\cA,\cC)$; 
  \item[(b)] $V$ is a discrete-time Lyapunov function for $\cA$ in $\cD$, in the sense that $\dom V\supset \cD\cup G(\cD)$ and there exist  $\alpha_1,\alpha_2\in \cK_\infty$, $\rho \in \cPd$  satisfying
\begin{subequations}
\begin{eqnarray}\label{eq:jumpcond}
  \alpha_1(|x|_\cA)\leq V(x)\leq\alpha_2(|x|_\cA),\;\;\;&\forall\,x\in \cD\cup G(\cD)\label{eq:BoundsJump}\\
V(g)-V(x)\leq-\rho(|x|_\cA) \;\;\;\;\; &\forall x\in \cD,\ g\in G(x). \label{eq:DecreadeHybridJump}
\end{eqnarray}
\end{subequations}
\end{itemize} 
Then $\cA$ is UGAS for hybrid system \eqref{eq:hybsys}. Moreover, if $V\in \cLL^\circ(\cA;\cC)$ and condition (b) is satisfied with $\rho\equiv 0$, then $\cA$ is UGS for hybrid system \eqref{eq:hybsys}.
\end{theorem} 

\begin{proof}
Consider any solution of~\eqref{eq:hybsys} $\psi:\dom \psi\to \R^n$. Consider any $j\in \N$ such that the $j$-th interval of flow, $I^j$, has nonempty interior. In this case, considering the restriction  $\psi(\cdot,j):I^j\to \cC$, we have $\psi(\cdot,j)\in \cS_{F,\cC}$. Applying Theorem~\ref{theorem:MainTheo} we have that $V\circ \psi(\cdot, j):I^j\to \R$ is strictly decreasing, with the rate determined by $\rho$ in \eqref{eq:LyapunovDecrease}. Conditions~\eqref{eq:jumpcond} guarantee the decrease  of $V\circ \psi$ during jumps of $\psi$, with the rate determined by $\rho$ in \eqref{eq:jumpcond}. Thus $\cA$ is UGAS,  following the same steps as in the proof of \cite[Theorem 3.18]{GoebSanf12}. The case with $V\in \cLL^\circ(\cA;\cC)$ and $\rho\equiv 0$ ensuring UGS of $\cA$ for system~\eqref{eq:hybsys} is straightforward. \qed
\end{proof}

\begin{remark}\label{rmk:weakercond}
There is no loss of generality in using the same class $\cK_\infty$ functions $\alpha_1,\alpha_2$ for positive definiteness in $\cC$ and in $\cD\cup G(\cD)$: if they were different, one considers the point-wise minimum for the lower bound and the point-wise maximum for the upper bound. The same reasoning applies to $\rho$.
Thus merging (L1) of Definition \ref{def:nicedefinition} and~\eqref{eq:BoundsJump} of Theorem \ref{proposition:StabHybrid} yields $\alpha_1,\alpha_2\in \cK_\infty$ such that 
\begin{equation}\label{eq:LyapUnifiedBound}
\alpha_1(|x|_\cA)\leq V(x)\leq\alpha_2(|x|_\cA)\;\;\forall x\in\cC\cup \cD\cup G(\cD),
\end{equation}
consistently with the hypothesis of~\cite[Theorem 3.18]{GoebSanf12}.\hfill $\triangle$
\end{remark}

\subsection{Homogeneous Hybrid Dynamics}
In this section we study a class of hybrid dynamical systems given by
\begin{equation}\label{eq:HomHybridSys}
\begin{cases}
\dot x=A_F x,\,\;\;\;\;\;x\in \cC=\{x\in \R^n\,\vert\,x^\top Q_F x\geq 0\},\\
x^+=A_Jx,\; \;\; x\in \cD=\{x\in \R^n\,\vert\,x^\top Q_J x\geq 0\},
\end{cases}
\end{equation}
where $A_F,A_J\in \R^{n\times n}$ and $Q_F,Q_J\in \Sym(\R^n):=\{S\in \R^{n\times n}\;\vert\;S=S^\top\}$. Such \eqref{eq:HomHybridSys} satisfies the \emph{hybrid basic conditions} defined in~\cite[Assumption 6.5]{GoebSanf12}. Moreover, $\cC$ and $\cD$ are symmetric cones, that is, if $x\in \cC$ ($x\in \cD$ resp.) then $\lambda x\in \cC$ ($\lambda x\in \cD$ resp.), for all $\lambda \in \R$. Noting that the flow and jump maps are linear, system~\eqref{eq:HomHybridSys} is \emph{homogeneous  with respect to the standard dilation}, as defined in~\cite{GoebTeel10} (see also~\cite[Chapter 9]{GoebSanf12}). Consistently with the converse result in~\cite{TunaTeel06}, we consider (proper piecewise $\cC^1$) candidate Lyapunov functions \emph{homogeneous of degree 2}; in particular we have the following definition.
\begin{definition}[Proper Piecewise Quadratic Functions]
A proper piecewise $\cC^1$ function $V:\R^n\to \R$ (recall Definition~\ref{def:patchy}) is a \emph{proper piecewise quadratic function} ($V\in \cPQ(\R^n))$ if 
\begin{itemize}[leftmargin=*]
\item For each $i\in \cI$, there exists $R_i\in \Sym(\R^n)$ such that $\cX_i=\{x\in \R^n\;\vert\;x^\top R_i x\geq 0\}$;
\item For each $i\in \cI$ there exists $P_i\in\Sym(\R^n)$ such that $V_i(x):=x^\top P_i x$, for all $x\in \R^n$.\hfill$\triangle$
\end{itemize} 
By definition of $\{\cX_i\}_{i\in \cI}$ and $\{V_i\}_{i\in \cI}$, a proper piecewise quadratic function is in particular \emph{even} and \emph{homogeneous of degree 2}, that is,
\[
V(\lambda x)=\lambda^2 V(x),\;\;\forall \,\lambda\in \R,\;\;\forall\;x\in \R^n.
\]
\end{definition}
With this definition, we now state a useful corollary of Theorem~\ref{proposition:StabHybrid} in the context of hybrid dynamical systems~\eqref{eq:HomHybridSys}.
\begin{corollary}\label{corollary:HomogeneusSystems}
Consider system~\eqref{eq:HomHybridSys}, with $Q_F$ not negative semi-definite. Suppose that there exist $V\in \cPQ$ and $\lambda_1,\lambda_2>0$  such that
\begin{equation}\label{eq:HomogeneLyapBounds}
\lambda_1|x|^2\leq V(x)\leq \lambda_2 |x|^2, \;\;\;\forall\;x\in \cC\cup\cD\cup A_J(\cD),
\end{equation}
and, for all $x\in \R^n$, for all $i\in \cI$, it holds that
\begin{equation}\label{eq:HomogLyapunFlow}
x^\top Q_Fx> 0\;\wedge\;x^\top R_ix> 0\;\Rightarrow x^\top P_iA_Fx <0.
\end{equation}
Moreover, suppose that for all $x\in \R^n$, for all $(j,i)\in \cI\times \cI$, it holds that
\begin{equation}\label{eq:HomgLyapuJump}
\begin{aligned}
x^\top Q_Jx\geq 0\;\wedge\;x^\top R_jx&\geq 0\;\wedge\; x^\top A_J^\top R_i A_Jx\geq 0\\&\Downarrow\\ x^\top A_JP_i A_Jx&-x^\top P_j x <0.
\end{aligned}
\end{equation}
Then the origin is UGAS for hybrid system~\eqref{eq:HomHybridSys}.
\end{corollary}
\begin{proof}[Sketch of the proof]
The hypothesis of $Q_F$ being not negative semi-definite ensures that $\overline{\inn(\cC)}=\cC=\{x\in \R^n\;\vert\;x^\top Q_F x\geq 0\}$, as required in Theorem~\ref{proposition:piecwiseprop}.
It is clear that~\eqref{eq:HomogeneLyapBounds} implies~\eqref{eq:LyapUnifiedBound} in Remark~\ref{rmk:weakercond}. Implication~\eqref{eq:HomogLyapunFlow} ensures  that $V\in \cLL(\{0\};\cC)$, since it implies condition~\eqref{eq:piecewiseFlowcond} of Theorem~\ref{proposition:piecwiseprop}.  Implication~\eqref{eq:HomgLyapuJump} ensures~\eqref{eq:DecreadeHybridJump}. Applying Theorem~\ref{proposition:StabHybrid} we conclude that $\{0\}$ is UGAS for system~\eqref{eq:HomHybridSys}.\qed
\end{proof}
\begin{remark}
We note that condition~\eqref{eq:HomogeneLyapBounds} does \emph{not} necessary imply $P_i>0$ for all $i\in \cI$. In fact it is sufficient to ensure that the functions $V_i(x)=x^\top P_i x$ are positive definite in their region of activation $\cX_i=\{x\in \R^n\;\vert\;x^\top R_i x\geq 0\}$, for each $i\in \cI$. The overall function $V$ then satisfies bounds as in~\eqref{eq:LyapUnifiedBound}; the quadratic bounds in~\eqref{eq:HomogeneLyapBounds} can be obtained by homogeneity of $V$.  Moreover, Corollary~\ref{corollary:HomogeneusSystems} is particularly appealing because conditions \eqref{eq:HomogeneLyapBounds},~\eqref{eq:HomogLyapunFlow} and~\eqref{eq:HomgLyapuJump} could be reduced, via S-Procedure, as a system of LMIs, but paying a price in terms of conservatism, see for example~\cite[Section 2.6.3]{Boyd94}. \hfill $\triangle$
\end{remark}

\begin{example}[Clegg Integrator]\label{example:clegg}
The Clegg integrator connected to an integrating plant has been shown to overcome intrinsic limitations of linear feedback \cite{BekHol01}, (see also \cite{PriQui18}). More specifically, using the $\varepsilon$-regularization suggested in \cite{NesTeel2011}, we focus on the hybrid closed-loop
\begin{equation}\label{eq:clegg}
\begin{cases}
\dot x=A_F x,\,\;\;\;\;\;x\in \cC=\{x\in \R^2\,\vert\,x^\top Q x\geq 0\},\\
x^+=A_Jx,\; \;\; x\in \cD=\{x\in \R^2\,\vert\,x^\top Q x\leq 0\},
\end{cases}
\end{equation}
with \[
A_F=\begin{bmatrix}
\quad 0 & \quad 1 \\{-1 }& \quad0
\end{bmatrix},\;\;A_J=\begin{bmatrix}
1 & 0 \\0 & 0
\end{bmatrix},\;\;Q=\begin{bmatrix}
1 & -\frac{1}{2 \varepsilon} \\-\frac{1}{2 \varepsilon} & 0
\end{bmatrix},
\]
and $\varepsilon>0$ being a small regularization parameter.

Following \cite{ZacNes11}, there does not exist a \emph{quadratic} Lyapunov function. In fact, given any symmetric and positive definite matrix
$P=\begin{bmatrix}
p_{11} & p_{12}\\p_{12} & p_{22}
\end{bmatrix}$,
and considering points $z_1=(-1,0)^\top$ and $z_2=(0,1)^\top$, we have that $z_1,z_2\in \cC$ and the Lyapunov inequalities
$z_1^\top P A_Fz_1<0$ and $z_2^\top P A_Fz_2<0$
would imply $p_{12}<0$ and $p_{12}>0$ respectively and hence, a contradiction.
UGAS of \eqref{eq:clegg} was established with nonconvex numerical piecewise quadratic constructions in \cite{ZacNes11}, and then via a nonconvex analytic construction in \cite{NesTeel2011}. More generally, various locally Lipschitz and piecewise quadratic Lyapunov functions for~\eqref{eq:clegg} satisfying~\eqref{eq:ClarkeLyapunov} are proposed in~\cite{ZacNesTee05},~\cite{NesZac05} and~\cite{NesTeel2011}. In these works, to avoid the complications of having non-differentiable points at the boundary of $\cC$, a different patching-technique is adopted. In particular, the common approach of these results requires considering an ``inflated counterpart'' of the flow set $\cC$, paying the price of a more convoluted analysis. To provide simplified analysis, we illustrate Corollary~\ref{corollary:HomogeneusSystems} by building three proper piecewise quadratic Lyapunov functions, one of them convex. For further details, see our conference paper~\cite{DelGoe19}. In what follows, we fix $\varepsilon=0.1$, but the functions that we construct work for any $\varepsilon$ such that $0<\varepsilon\leq0.1$.

\emph{Max of 2 sign-indefinite quadratics.} Let 
 \begin{equation}\label{eq:CleggLyapFunc}
V_M(x):=\max\{x^\top P_1 x, x^\top P_2x\}, 
\end{equation}
\[
\text{with}\;P_1=\begin{bmatrix}
1 & -0.1 \\-0.1 & 0.5
\end{bmatrix}
 \;,\;
P_2=\begin{bmatrix}
2.5& \;1.4\\1.4 & \;0.5
\end{bmatrix}.
\]
Note that $P_2$ is not positive definite (see also \cite{LinLi17}).
It can be seen that all the conditions of Corollary~\ref{corollary:HomogeneusSystems} are satisfied by $V_M$ and we can conclude that $\cA=\{0\}$ is UGAS for system \eqref{eq:clegg} as shown in \cite{DelGoe19}. See Fig. \ref{fig:clegg}(a) for a graphical representation of our construction, where nonconvexity of $V_M$ emerges from the fact that $P_2$ is not sign-definite.

\emph{Mid of quadratics.} Consider the symmetric matrices
\[
P_1=\begin{bmatrix}
1 & \;0.25 \\0.25 & \;0.7
\end{bmatrix}
,\;\;
P_2=\begin{bmatrix}
0.55& \;-0.2\\-0.2 & \;0.25
\end{bmatrix} ,\;\;
P_3=\begin{bmatrix}
\frac{25}{16}& \;\frac{49}{160}\\\star& \;0.25
\end{bmatrix} 
\]
and consider the function
\begin{equation}\label{eq:cleggMIdLyap}
\begin{aligned}
V_{\text{mid}}(x)&:=\text{mid}\{V_1(x),V_2(x),V_3(x)\}\\
&:=\max\{\min\{V_1,V_2\},\min\{V_2,V_3\},\min\{V_1,V_3\}\},
\end{aligned}
\end{equation}
where $V_i(x):=x^\top P_i x$, that is a max-min function defined  as in~\eqref{eq:Mmm}. Intuitively, the  \virgolette{mid} operator selects the function whose value lies between the two others ones. Again, it can be seen that all the conditions of Corollary~\ref{corollary:HomogeneusSystems} are satisfied and we conclude again that $\cA=\{0\}$ is UGAS for the system \eqref{eq:clegg}. See Fig. \ref{fig:clegg}(b) for a graphical representation of our construction, which shows again nonconvex level sets of $V_{\text{mid}}$.
\begin{figure}
  \centering
  \begin{tabular}{@{}c@{}}
    \includegraphics[width=.325\linewidth,height=140pt]{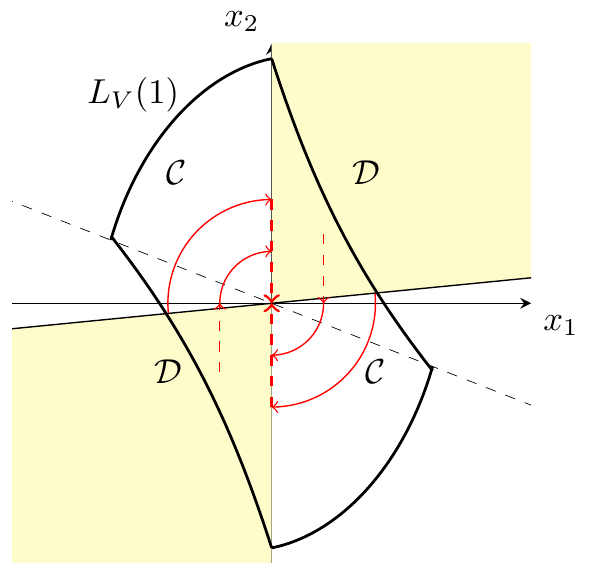} \\[\abovecaptionskip]
    {\small (a) $V_M$ defined in \eqref{eq:CleggLyapFunc}.}
  \end{tabular}
  \begin{tabular}{@{}c@{}}
    \includegraphics[width=.325\linewidth,height=140pt]{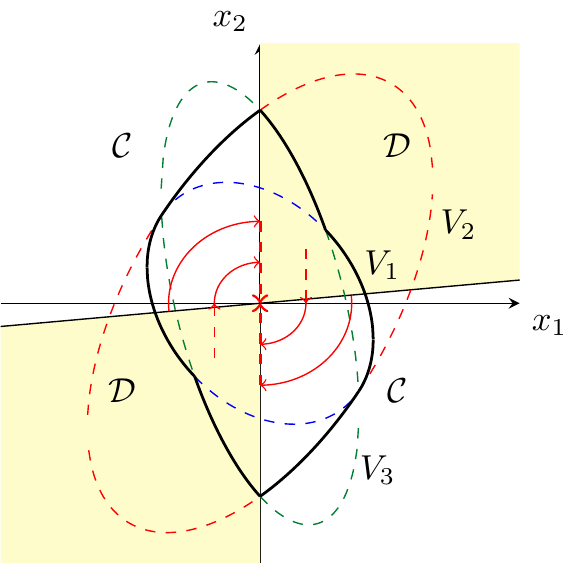} \\[\abovecaptionskip]
    \small (b) $V_{\text{mid}}$ defined in \eqref{eq:cleggMIdLyap}.
  \end{tabular}
\begin{tabular}{@{}c@{}}
    \includegraphics[width=.325\linewidth,height=140pt]{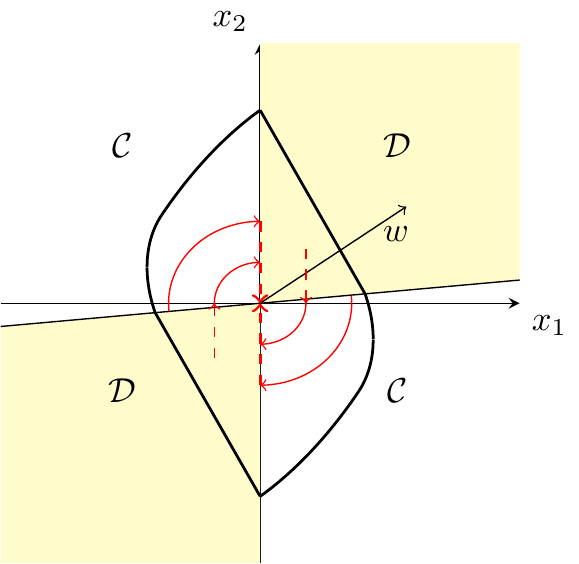} \\[\abovecaptionskip]
    \small (c) $V_{\text{conv}}$ defined in \eqref{eq:ConvCleggFun}.
    \end{tabular}
\caption{Level sets of the constructed Lyapunov functions for system \eqref{eq:clegg}; in red some  particular solutions.}
    \label{fig:clegg}
\end{figure}

\emph{Convex Lyapunov function.} The Lyapunov functions above are both nonconvex. We construct here a convex one, starting from $V_{\text{mid}}$. Looking at the level set $L_{V_{\text{mid}}}(1)$, the idea is to connect the points of intersection of $L_{V_{\text{mid}}}(1)$ with the two lines that form the boundary of $\cD$ using a straight line. We thus define
\begin{equation}\label{eq:ConvCleggFun}
V_{\text{conv}}(x)=\begin{cases}
V_{\text{mid}}(x),\;\;&\text{if }x\in \cC,\\
\inp{w}{x}^2,\;\;&\text{if }x\in \cD,
\end{cases}
\end{equation}
where $w=(0.9574,0.7071)^\top$ is a vector tangent to the line of interest, suitably scaled to ensure continuity. This function satisfies the conditions of Corollary~\ref{corollary:HomogeneusSystems} from the properties of $V_{\text{mid}}$. It is represented in Fig. \ref{fig:clegg}(c).

It is easy to see that for all three Lyapunov functions $V_M,V_{mid},V_{conv}:\R^2\to \R$, the Clarke conditions 
in~\eqref{eq:ClarkeLyapunov} are not satisfied on the line $\cR:=\{x=(x_1,x_2)^\top \in \R^2\;\vert\;x_1=0\}\subset \bd(\cC)$. More precisely, for all $x\in \cR$, 
\[
\exists v\in \partial V(x)\;\text{such that}\; v^\top A_Fx>0,
\]
where $V$ denotes any of the Lyapunov functions $V_M,V_{mid}$ or $V_{conv}$.
The observation that~\eqref{eq:ClarkeLyapunov} may fail to hold when patching $V$ at the boundary of $\cC$ had been already made in~\cite[Remark 8]{NesTeel2011}, where a more convoluted patching is adopted.
Our Theorem~\ref{proposition:StabHybrid} (and Corollary~\ref{corollary:HomogeneusSystems}) provides a useful answer to that concern, showing that even though~\eqref{eq:ClarkeLyapunov} cannot be used to establish UGAS with the intuitive patching on $\bd(\cC)$, UGAS can still be concluded from the results of this paper.\hfill $\triangle$
\end{example}

\section{Conclusions}
In this work, we studied the problem of stability for a class of differential inclusions, with a particular interest in applications to hybrid dynamical systems. We provided sufficient Lyapunov conditions for a particular class of piecewise-defined locally Lipschitz functions, requiring to check the Lyapunov decrease inequality only on a dense subset of a given domain of interest $\cC$. We then studied the relations between our approach and the existing literature on locally Lipschitz Lyapunov functions, showing that our conditions are in general less restrictive than the Clarke gradient-based conditions.
We finally applied our approach in the context of hybrid dynamical systems, with particular care to the homogeneous case. Several examples are provided to show the novelty and the usefulness of our results.

\bibliographystyle{plain}
\bibliography{biblio}

\end{document}